\documentclass[12pt]{amsart}
\usepackage[margin=1in]{geometry}
\usepackage{titlesec}
\titleformat{\section}{\normalfont\normalsize\bfseries}{\thesection}{1em}{}
\usepackage{fancyhdr}
\usepackage{hyperref}
\usepackage{float}
\usepackage{xcolor}
\usepackage{cite}
\usepackage{amsmath,amsthm,amsfonts,amssymb,bm}
\usepackage{setspace}
\usepackage{mathtools}
\usepackage{ulem}
\newtheorem{theorem}{Theorem}[section]

\newtheorem{definition}{Definition}[section]
\newtheorem{lemma}{Lemma}[section]

\numberwithin{equation}{section}
\numberwithin{theorem}{section}

\makeatletter
\def\tagform@#1{\maketag@@@{\ignorespaces#1\unskip\@@italiccorr}}
\let\orgtheequation\theequation
\def\theequation{(\orgtheequation)}
\makeatother

\allowdisplaybreaks
\everymath{\displaystyle}
\begin{document}
	\title{finite time blow-up analysis for the generalized Proudman-Johnson model}
	\begin{abstract}
		In this paper, we study the generalized Proudman-Johnson equation posed on the torus. In the critical regime where the parameter $a$ is close to and slightly greater than 1, we establish finite time blow-up of smooth solutions to the inviscid case. Moreover, we show that the blow-up is asymptotically self-similar for a class of smooth initial data. In contrast, when the parameter $a$ lies slightly below 1, we prove the global in time existence for the same initial data. In addition, we demonstrate that inviscid Proudman-Johnson equation with H\"{o}lder continuous data also develops a self-similar blow-up. Finally, for the viscous case with $a>1$, we prove that smooth initial data can still lead to finite time blow-up.
	\end{abstract}

	\author{Jie Guo$^{*}$\, \,  QuanSen Jiu$^{\dagger}$}
	
	\address{$^{*}$School of Mathematical Sciences, Capital Normal University, Beijing, 100048, P.R. China.}
	
	\email{2230501023@cnu.edu.cn}
	\address{$^{\dagger}$School of Mathematical Sciences, Capital Normal University, Beijing, 100048, P.R. China.}
	
	\email{jiuqs@cnu.edu.cn}

	\maketitle
	
	\medskip

	
	\section{Introduction}
	In this paper, we investigate a one-parameter family of partial differential equations given by
	\begin{align}\label{gpje}
		u_{txx} + uu_{xxx} = au_{x}u_{xx}+\nu u_{xxxx}, \quad t>0, \ x\in\mathbb{T},
	\end{align}
	posed on the torus \(\mathbb{T} = [-\pi,\pi]\), where $a\in\mathbb{R}$ is a parameter and $\nu\ge0$ represents the viscosity. This equation was originally introduced in \cite{OHZ2000} for specific values of $a$ as a model derived from high-dimensional Navier-Stokes equations under certain symmetry assumptions. By defining $u_{xx}=\omega$, equation \ref{gpje} can be written as
	\begin{align}\label{PJE}
		\begin{cases}
			\omega_{t} + u\omega_{x} = a\omega u_{x}+\nu\omega_{xx}, & t >0, \ x \in \mathbb{T}, \\
			u_{xx} =\omega, & t >0, \ x \in \mathbb{T}.
		\end{cases}
	\end{align}
	When $a=1$, \eqref{PJE} is the Proudman-Johnson equation. In particular, when $a=1$ and $\nu=0$, \eqref{PJE} is the inviscid Proudman-Johnson equation, which models inviscid, incompressible flows near a boundary (see \cite{CS1989}). For general $a\in \mathbb{R}$, equation \eqref{PJE}, which is equivalent to \eqref{gpje}, is called a generalized Proudman-Johnson equation. By generalizing the Proudman-Johnson equation with a parameter
	  $a$, one can systematically study the balance between the ``convection term''($uu_{xxx}$) and the ``stretching term''($u_{x}u_{xx}$), which leads to either the creation or depletion of singularities in finite time (see \cite{OO2005, OS2008}).

 The generalized Proudman-Johnson equation encompasses several well-known models. When $a=-3$, it reduces to the classical Burgers' equation, which plays a central role in gas dynamics (see \cite{B1948}). For $a=-2$, the equation becomes the Hunter-Saxton equation (see \cite{B1991, L2007, L2008, Y2004}), which arises in the modeling of the orientation dynamics of nematic liquid crystals. The case $a=0$, the equation $u_{txx} + uu_{xxx} = 0$ appears in contexts that connect projective geometry and gravitational models (see \cite{P2001}). The inviscid form of model \eqref{PJE} also resembles the generalized Constantin-Lax-Majda (gCLM) model (see \cite{OS2008}), which takes the form
		\begin{align}\label{gCLM}
		\begin{cases}
			\omega_{t} + u\omega_{x} = a\omega u_{x}+\nu\omega_{xx}, & t >0, \ x \in \mathbb{T}, \\
			u_{x} = H\omega, & t >0, \ x \in \mathbb{T},
		\end{cases}
	\end{align}
	where $H$ represents the Hilbert transform. When $a=1$, the gCLM model \eqref{gCLM} reduces to the De Gregorio (DG) equation. Notably, the explicit steady state $(\omega,u) = (-\sin (x),\sin(x))$ is a solution to both  equation \eqref{PJE} and the DG equation on the torus $\mathbb{T}$.

The local well-posedness of the generalized  Proudman Johnson equation \ref{gpje} in the periodic setting was established in \cite{OHZ2000, OH2009}. In particular, \cite{OHZ2000} derived a sufficient condition for finite time blow-up of solutions in the case $a=\infty$,
	\begin{align*}
		u_{txx} - u_{x}u_{xx} = \nu u_{xxxx}, \quad t>0, \ x \in \mathbb{T}.
	\end{align*}
	Moreover, the authors in \cite{OHZ2000, OH2009} investigated parameter conditions on $a$ that ensure the global existence of solutions for general initial data. Further analytical and numerical results concerning blow-up phenomena were also presented in \cite{OHZ2000}. For $a=1$, more refined blow-up criteria have been derived, as discussed in \cite{CS1989}, where the authors employed various trajectory-based analytical techniques.

	  In \cite{CW2010}, the authors proved the existence of a class of global-in-time solutions to the generalized inviscid Proudman-Johnson equation for parameters of the form $a = -\frac{n+3}{n+1}$, where $n\in\mathbb{N}$. Subsequently, in \cite{CW2012}, global existence was established for $a\in [-2,-1)$ in the inviscid setting, using the method of characteristics. In \cite{SS2013,SS2015}, various global existence and blow-up scenarios are investigated based on a representation formula for the velocity gradient along particle trajectories. In contrast, singularity formation for $a>1$ remains less understood. In \cite{OH2009}, the authors construct non-smooth self-similar solutions of the form $u(x,t)=\frac{F(x)}{T-t}$, which blow-up at time $T$. These solutions are obtained by solving the nonlinear ordinary differential equation
	\begin{align}\label{Fequ}
		F''(x) + F(x)F'''(x) - aF'(x)F''(x) = 0, \quad x\in \mathbb{T},
	\end{align}
	for $a>1$. However, in this regime, equation \eqref{Fequ} admits only non-smooth solutions. More recently, Kogelbauer \cite{KF2020} established new criteria for global existence and finite time blow-up in the inviscid generalized Proudman-Johnson equation, depending on the parameter $a$. In particular, \cite{KF2020} demonstrated that even smooth initial data may lead to finite time singularities when $a>1$. Additionally, \cite{KF2020} provides a physical derivation of the generalized equation for arbitrary values of the parameter $a$.

	In this paper, we will investigate model \ref{PJE} on the torus $\mathbb{T}$ for parameter values $a$ near 1, which can be interpreted as a slight perturbation of the Proudman-Johnson case $(a=1)$. For the inviscid model \ref{PJE}, this regime is particularly interesting for smooth initial data, as it captures two distinct dynamical scenarios:
	\begin{itemize}
		\item When $a>1$, the advection term is slightly weaker than the vortex stretching term, leading to finite time blow-up;
		\item When $a<1$, the advection term slightly dominates the vortex stretching term, resulting in different long-time behavior.
	\end{itemize}
	  Moreover, we establish finite time singularity formation for the inviscid model \ref{PJE} with  H\"{o}lder continuous  data in the case $a=1$ and the finite-time blow-up for the viscous case with $a>1$.
	
	Before stating our main results, we introduce a class of weighted Sobolev spaces that will serve as the functional framework for our analysis.
	\begin{definition}[Weighted norms and spaces]
	  Define the singular weight $\rho = \frac{1}{4\pi\sin^{2} (\frac{x}{2})}$ and the weighted norms  $\|\cdot\|_{\mathcal{H}}$ and  $\|\cdot\|_{\mathcal{W}}$ as follows
	  \begin{align}
	  	\|f\|_{\mathcal{H}}^{2} = \frac{1}{4\pi}\int_{\mathbb{T}}\frac{|f_{x}|^{2}}{\sin^{2} (\frac{x}{2})} \, \mathrm{d}x,
	  \end{align}
	  \begin{align}
	  	\|f\|_{\mathcal{W}} = \|f\|_{\mathcal{H}}^{2} + \int_{\mathbb{T}} |f_{xx}|^{2}\cos^{2} (\frac{x}{2}) \, \mathrm{d}x.
	  \end{align}
	  The associated Hilbert spaces are defined as
	  \begin{align*}
	  	\mathcal{H} = \{f\in H^{1}(\mathbb{T}) |\, \text{f is odd},\ \|f\|_{\mathcal{H}} < + \infty\},
	  \end{align*}
	  \begin{align*}
	  	\mathcal{W} = \{f\in H^{2}(\mathbb{T}) |\, \text{f is odd},\ \|f\|_{\mathcal{W}} < +\infty\},
	  \end{align*}
	  which are equipped with inner products naturally induced by the corresponding norms.
	\end{definition}
	Define $\tilde{e}_{k}^{(o)}=\frac{\sin((k+1)x)}{k+1}-\frac{\sin (kx)}{k}$. In fact, $\{\tilde{e}_{k}^{(o)},k\geq 1\}$ forms a complete orthonormal basis of $\mathcal{H}$ (see Lemma \ref{le cobasis}).
   The norm and inner product in $\mathcal{H}$  were first introduced in \cite{LL2020} for stability analysis. The $\mathcal{W}$ norm can be found in \cite{C2021} for the purpose of higher order derivative estimates.

    Throughout this paper,  $\|\cdot\|_{L^{p}},\|\cdot\|_{H^{m}}$ and $\|\cdot\|_{L^{\infty}}$ represent the norms in the spaces $L^{p}(\mathbb{T}), H^{m}(\mathbb{T})$ and $L^{\infty}(\mathbb{T})$, respectively. Similarly, $\left\langle\cdot, \cdot\right\rangle$ denotes the usual inner product in $L^{2}(\mathbb{T})$, defined as
   \begin{align*}
   	\left\langle f, g\right\rangle = \int_{\mathbb{T}}fg \, \mathrm{d}x.
   \end{align*}
   We use $C,C_{i}$ to denote absolute constants and $C(A,B,\cdots,Z)$ to denote constants depending on $A,B,\cdots,Z$. And we also employ the notation $A\lesssim B$ to indicate that there exists a constant $C$ such that $A\leq CB$.

Now we are ready to state our main results.   Our first main result establishes the existence of a family of self-similar solutions to the inviscid model \ref{PJE}.
	\begin{theorem}\label{them profile}
		{\bf (The  case $\nu=0$)} There exists an absolute constant $\delta_{1}>0,$  sufficiently small, such that the following statements hold:
		
		(1) For \(a\in (1, 1+\delta_{1})\) The model \ref{PJE} with $\nu=0$ develops a finite time singularity for some \(C^{\infty}\) initial data;
		
		(2) For all $a\in(1-\delta_{1},1+\delta_{1}),$ the model \ref{PJE} admits a self-similar solution of the form
		\begin{align*}
			\omega(x,t) = \frac{1}{1 + c_{\omega,a}t}\omega_{a}(x),
		\end{align*}
		where $\omega_{a}$ is an odd profile and $c_{\omega,a}$ is the scaling parameter,  satisfying
		\begin{align}\label{the estimates}
			\|\omega_{a} +\sin (x)\|_{\mathcal{W}} \lesssim |1-a|, \quad |c_{\omega,a} - (1-a)|\leq \min\{C|1-a|^{2}, |1-a|\},
		\end{align}
		for some absolute constant $C > 0$. More precisely:
		\begin{itemize}
			\item for $1 < a < 1+\delta_{1}$, the scaling parameter satisfies $c_{\omega,a} < 0$, and the corresponding solution $\omega(x,t)$ blows up in finite time $T = -\frac{1}{c_{\omega,a}}$;
			\item for $a = 1$, the explicit steady state $\omega_{1} = -\sin (x)$ solves the system with $c_{\omega,a} = 0$;
			\item for $1-\delta_{1}<a<1$, one has $c_{\omega,a} > 0$, and the solution $\omega(x,t)$ exists globally with $O(t^{-1})$ decay rate as $t\to\infty$.
		\end{itemize}
	\end{theorem}

	Our next result concerns the formation of finite time singularities in the generalized Proudman-Johnson model with H\"{o}lder continuous initial data.
	\begin{theorem}\label{the holder}
		{\bf (The  case $\nu=0$ and $a=1$)}  There exists a constant $\delta_{2}>0$, such that for all $\alpha\in(1-\delta_{2},1),$ the model \ref{PJE} with $\nu=0$ and $a=1$ develops a finite time singularity for some $C^{\alpha}$ initial data. Moreover, there exists a $C^{\alpha}$ self-similar profile, analogous to the setting in Theorem \ref{them profile}.
	\end{theorem}
	
	The final result concerns the finite time blow-up of model \ref{PJE} in the presence of viscosity. The dynamic rescaling formulation suggests that the viscous terms are asymptotically small. Building upon Theorem \ref{them profile} in the regime $1<a<1+\delta_{3},$ we establish the following result.
	
	\begin{theorem}\label{the viscous}
		{\bf (The  case $\nu>0$)} There exists a constant $\delta_{3}>0$ such that for all $a\in(1,1+\delta_{3}),$ the model \ref{PJE} with $\nu>0$ develops a singularity in finite time for some $C^{\infty}$ initial data.
	\end{theorem}
	We emphasize that, in contrast to the inviscid case, exact self-similar blow-up profiles do not exist due to the presence of viscosity. 	
	
	In our analysis, we employ the framework of dynamic rescaling to establish the formation of singularities. This formulation was first introduced by McLaughlin, Papanicolaou, and co-authors in their study
	of self-similar blow-up of the nonlinear Schr\"{o}dinger equation (see \cite{MP1986,LP1988}). It was later developed into a powerful modulation technique, and has been applied to various blow-up problems including the nonlinear Schr\"{o}dinger equation \cite{KM2006,MR2005}, the nonlinear heat equation \cite{MZ1997}, the generalized KdV equation \cite{MM2014}, and other
	dispersive problems. More recently, this approach has been successfully adapted to prove singularity formation in the gCLM models \cite{C2021, CHH2021}, the Euler equations \cite{CH2021, E2021}, and the Hou-Li model \cite{HW2024}.
	
	Our blow-up analysis for the inviscid model \ref{PJE} consists of several steps. In the first step, we reformulate the original equation \eqref{PJE} via dynamic rescaling, thereby transforming the singularity formation problem into the stability analysis of an approximate steady state in the rescaled variables. In the second step, we perform the stability analysis of this approximate steady state in the rescaled formulation, which primarily involves:
	\begin{itemize}
		\item the identification of an approximate steady state;
		\item energy estimates in a singularly weighted norm, providing both linear and nonlinear stability of this approximate profile.
	\end{itemize}
	By tracing back to the original variables, this yields finite time blow-up for the inviscid model \eqref{PJE}. In the third step, we establish convergence of the dynamically rescaled solution to the true self-similar profile. However, the blow-up analysis for model \ref{PJE} with viscosity is more challenging, as the viscous term does not provide damping and generates some bad terms when a singular weighted norm is employed. To address this, we establish the blow-up analysis for the viscous \ref{PJE} using an energy norm that combines a singular weighted energy norm with a sum of higher-order Sobolev norms, as employed in \cite{HW2024}.

	The paper is organized as follows. In Section \ref{section pre}, we present some preliminaries and useful lemmas. Section \ref{section COSSP} is devoted to the construction of a family of self-similar profiles for \ref{PJE} when $a$ is close to 1, together with their stability analysis and the proof of Theorem \ref{them profile}. Section \ref{section holder} contains the proof of Theorem \ref{the holder}, concerning finite-time blow-up from H\"{o}lder continuous initial data for the model \ref{PJE} in the case $a=1$. Finally, in Section \ref{section viscous}, we prove Theorem \ref{the viscous} through a special energy norm to estimate the viscous terms.
	
	\section{Preliminaries}\label{section pre}
	In this section, we collect several technical lemmas needed later. We first derive an explicit expression for the velocity field $u$ in terms of the vorticity $\omega$, under the assumption of odd symmetry. The following lemma is obtained through straightforward integration by parts.
	\begin{lemma}\label{le psi expression}
		Suppose that $\omega,u$ are odd, $2\pi$-periodic functions on $[-\pi,\pi]$, satisfying $u_{xx}=\omega.$ Then $u$ can be expressed as
		\begin{align}\label{psi expression}
			u(x) = \int_{0}^{x}(x-y)\omega(y) \, \mathrm{d}y + xu_{x}(0),
		\end{align}
		where $u_{x}(0)$ is given by
		\begin{align*}
			u_{x}(0) = \frac{1}{2\pi}\int_{0}^{2\pi}y\omega(y) \, \mathrm{d}y.
		\end{align*}
	\end{lemma}

	\begin{lemma}\label{le sinfx}
		Let $\rho^{1/2} f \in L^{2}(\mathbb{T})$. Then the following identity holds.
		\begin{align*}
			\left\langle\sin(x) f_{x},f\rho\right\rangle = \frac{1}{2}\left\langle f^{2},\rho\right\rangle.
		\end{align*}
	\end{lemma}
	\begin{proof}[Proof of Lemma \ref{le sinfx}]
		By a direct integration by parts, it holds
		\begin{align*}
			\left\langle \sin(x) f_{x}, f \rho \right\rangle & = \frac{1}{4\pi} \int_{-\pi}^{\pi}\frac{\sin(x) (f^{2})_{x}}{2\sin^{2}(\frac{x}{2})} \, \mathrm{d}x = \frac{1}{4\pi}\int_{-\pi}^{\pi}\frac{\cos(\frac{x}{2})}{2\sin(\frac{x}{2})}(f^{2})_{x} \, \mathrm{d}x\\
			& = \frac{1}{4\pi}\int_{-\pi}^{\pi}\frac{f^{2}}{2\sin^{2}(\frac{x}{2})} \, \mathrm{d}x\\
			& = \frac{1}{2}\left\langle f^{2},\rho\right\rangle.
		\end{align*}
	\end{proof}
	
	Lemma \ref{le estimates} will be used in Section \ref{section nonlinear},\ref{section holder} and \ref{section viscous} to estimate the nonlinear terms.
	\begin{lemma}\label{le estimates}
		Let $\omega,u$ be odd, $2\pi$-periodic functions on $[-\pi,\pi]$ related by $u_{xx} = \omega$, with $\omega \in \mathcal{H}$. Then the following estimates hold.
		
		(1)
		\begin{align*}
			\|(u_{x} - u_{x}(0))\rho^{1/2}\|_{L^{2}} \lesssim \|\omega\|_{\mathcal{H}}.
		\end{align*}
		
		(2)
		\begin{align*}
			\|\omega\rho^{1/2}\|_{L^{\infty}}, \|u\rho^{1/2}\|_{L^{\infty}}, \|u_{x}\|_{L^{\infty}}\lesssim \|\omega\|_{\mathcal{H}}.
		\end{align*}
	\end{lemma}
	\begin{proof}[Proof of Lemma \ref{le estimates}]
		To prove (1), by applying Hardy-type inequality and Poincar\'{e}'s inequality, we obtain
		\begin{align*}
			\|(u_{x} - u_{x}(0))\rho^{1/2}\|_{L^{2}} & = \bigg\|\left(\int_{0}^{x}\partial_{x}^{2}u(y)\, \mathrm{d} y\right)\rho^{1/2}\bigg\|_{L^{2}}\\
			& \lesssim \|u_{xx}\|_{L^{2}} = \|\omega\|_{L^{2}}\\
			& \lesssim \|\omega_{x}\|_{L^{2}} = \|\omega_{x}\rho^{1/2}\rho^{-1/2}\|_{L^{2}}\\
			&\lesssim \|\omega\|_{\mathcal{H}}.
		\end{align*}
		
		To prove (2), since $\omega$ is odd and periodic, we only need to estimate $\|\omega\rho^{1/2}\|_{L^{\infty}}$ on $[0,\pi]$. Due to $\sin(\frac{x}{2}) \geq \frac{2}{\pi}x$ for $x\in[0,\pi]$, we have
		\begin{align*}
			\bigg|\frac{\omega}{\sin(\frac{x}{2})}\bigg| & \lesssim \bigg|\frac{1}{x}\int_{0}^{x}\partial_{x}\omega(y)\, \mathrm{d}y\bigg|\\
			& \lesssim \frac{1}{x}\left(\int_{0}^{x} \sin^{2}(\frac{y}{2})\, \mathrm{d}y \right)^{1/2} \bigg\|\frac{\omega_{x}}{\sin(\frac{x}{2})}\bigg\|_{L^{2}}\\
			&\lesssim \|\omega_{x}\rho^{1/2}\|_{L^{2}},
		\end{align*}
		which implies that
		\begin{align*}
			\|\omega\rho^{1/2}\|_{L^{\infty}} \lesssim \|\omega\|_{\mathcal{H}}.
		\end{align*}
		For $\|u\rho^{1/2}\|_{L^{\infty}}$ and $\|u_{x}\|_{L^{\infty}}$, we use  Pincar\'{e}'s inequality to yield
		\begin{align*}
			\bigg|\frac{u}{\sin(\frac{x}{2})}\bigg| & \lesssim \|u_{x}\|_{L^{\infty}} \lesssim \|u_{xx}\|_{L^{2}} = \|\omega\|_{L^{2}}\\
			&\lesssim \|\omega_{x}\|_{L^{2}} \lesssim \|\omega\|_{\mathcal{H}}.
		\end{align*}
		Thus, we obtain the desired bounds
		\begin{align*}
			\|u\rho^{1/2}\|_{L^{\infty}}, \|u_{x}\|_{L^{\infty}}\lesssim \|\omega\|_{\mathcal{H}}.
		\end{align*}
	\end{proof}
	
	The following lemmas will be applied in Section \ref{section holder}.
	\begin{lemma}\label{le 1-x}
		For $x\in[0,1],a,b>0$, it holds
		\begin{align*}
			(1-x^{a})x^{b}\leq \frac{a}{b}.
		\end{align*}
	\end{lemma}
	\begin{proof}[Proof of Lemma \ref{le 1-x}]
		Applying Young's inequality, we have
		\begin{align*}
			(1-x^{a})x^{b}&=\frac{a}{b}\left(\frac{b}{a}(1-x^{a})\right)(x^{a})^{\frac{b}{a}}\\
			&\leq \frac{a}{b}\left(\frac{\frac{b}{a}(1-x^{a})+\frac{b}{a}x^{a}}{1+\frac{b}{a}}\right)^{\frac{b}{a}+1}=\frac{a}{b}\left(\frac{\frac{b}{a}}{1+\frac{b}{a}}\right)^{\frac{b}{a}+1}\leq \frac{a}{b}.
		\end{align*}
	\end{proof}
	
	\begin{lemma}\label{le res estimates}
		Let $\kappa=\frac{7}{8}<\frac{9}{10}<\alpha<1.$ Then, for all $x\in[-\pi,\pi]$ and $i=0,1,2,$ the following estimates hold.
		\begin{align}\label{omegares estimate}
			|\partial_{x}^{i}\omega_{res}|\lesssim |\alpha-1||\sin (x)|^{\kappa-i},
		\end{align}
		\begin{align}\label{psires estimate}
			|\partial_{x}^{i}\psi_{res}| + |\partial_{x}^{i}\psi_{res,x}| \lesssim |\alpha-1|,
		\end{align}
		\begin{align}\label{cancellation}
			\begin{aligned}
				& \quad |(\alpha-1)\bar{\omega}_{\alpha,x}-\sin(x)\omega_{res,xx}|+|\sin(x) \partial_{x}\left((\alpha-1)\bar{\omega}_{\alpha,x}-\sin(x)  \omega_{res,xx}\right)|\\
				& \lesssim\min\{|\alpha-1|,|x|^{2}\}|\sin(x)|^{\alpha-1}.
			\end{aligned}
		\end{align}
	Here, we used the following notations
	\begin{align*}
		\bar{\omega}(x) & = -\sin(x), \quad
		\bar{\omega}_{\alpha}(x) = -\mathrm{sign}(x) \, |\sin(x)|^{\alpha}, \quad
		\omega_{res}(x) = \bar{\omega}_{\alpha}(x) - \bar{\omega}(x),
	\end{align*}
	and
	\begin{align*}
		\bar{u}(x) &= \sin(x), \quad
		\bar{u}_{\alpha,xx}(x) = \bar{\omega}_{\alpha}(x), \quad
		u_{res}(x) = \bar{u}_{\alpha}(x) - \bar{u}(x).
	\end{align*}
	\end{lemma}
	\begin{proof}[Proof of Lemma \ref{le res estimates}]
		  By symmetry, it suffices to consider the case $x\geq 0$. We begin with the proof of \ref{omegares estimate}. Applying Lemma \ref{le 1-x}, we obtain
		\begin{align*}
			|\omega_{res}| & = \big|\big(\sin(x)\big)^{\alpha}-\sin(x)\big|\\
			& = \big|\big(\sin(x)\big)^{\kappa}(\sin x)^{\alpha-\kappa}\big(1-(\sin x)^{1-\alpha}\big)\big|\\
			&\lesssim \frac{1-\alpha}{\alpha-\kappa} |\sin(x)|^{\kappa} \lesssim (1-\alpha)|\sin(x)|^{\kappa}.
		\end{align*}
		A straightforward computation gives
		\begin{align*}
			\bar{\omega}_{\alpha,x} = -\alpha \big(\sin(x)\big)^{\alpha-1} \cos(x)
		\end{align*}
		and
		\begin{align*}
			\bar{\omega}_{\alpha,xx} = -\alpha(\alpha-1) \big(\sin(x)\big)^{\alpha-2} \cos^{2}(x) + \alpha(\sin(x))^{\alpha}.
		\end{align*}
		Applying Lemma \ref{le 1-x}, we deduce that
		\begin{align*}
			|\bar{\omega}_{\alpha,x}-\bar{\omega}_{x}| & =|\cos(x) \big(1-\alpha \big(\sin(x)\big)^{\alpha-1}\big)|\\
			& \lesssim \bigg|\big(\sin(x)\big)^{\kappa-1} \big(\sin(x)\big)^{\alpha-\kappa} \bigg(1 - \alpha \big(\sin(x)\big)^{1-\alpha}\bigg)\bigg|\\
			& \lesssim |\alpha - 1| |\sin(x)|^{\kappa - 1},
		\end{align*}
		The estimate for $i=2$ in \eqref{omegares estimate} follows from a similar argument and is omitted. The bound \ref{psires estimate} is a direct consequence of Lemma \ref{le psi expression}.
		
		Next, we address the estimate \eqref{cancellation}. First, it follows from \ref{omegares estimate} that
		\begin{align*}
			|(\alpha-1)\bar{\omega}_{\alpha,x}-\sin(x) \omega_{res,xx}| \lesssim |\alpha-1| |\sin(x)|^{\alpha-1}.
		\end{align*}
		Then, by direct computation, we obtain
		\begin{align*}
			&|(\alpha-1)\bar{\omega}_{\alpha,x} - \sin(x)\omega_{res,xx}|\\
			= & |-\alpha(\alpha-1)\big(\sin(x)\big)^{\alpha-1}\cos(x) + \alpha(\alpha-1) \big(\sin(x)\big)^{\alpha-1} \cos^{2}(x) - \alpha \sin(x) \big(\sin(x)\big)^{\alpha} + \sin(x) \bar{\omega}_{xx}|\\
			\leq & |\alpha(\alpha-1) \big(\sin(x)\big)^{\alpha-1} \big(\cos(x) - \cos^{2}(x)\big)| + \alpha \big(\sin(x)\big)^{\alpha+1} + \sin(x) |\bar{\omega}_{xx}|\\
			\leq & \big(\sin(x)\big)^{\alpha-1} |x|^{2} + \big(\sin(x)\big)^{\alpha+1}\\
			\lesssim & \big(\sin(x)\big)^{\alpha-1}|x|^{2},
		\end{align*}
		which leads to
		\begin{align*}
			|(\alpha-1)\bar{\omega}_{\alpha,x}-\sin(x) \omega_{res,xx}| \lesssim |\alpha-1|^{1/2}|x|.
		\end{align*}
		Similarly, we get
		\begin{align*}
			|\sin(x) \partial_{x} \big((\alpha-1)\bar{\omega}_{\alpha,x} -\sin(x) \omega_{res,xx}\big)| \lesssim \min\{|\alpha-1|, |x|^{2}\} |\sin(x)|^{\alpha-1}.
		\end{align*}
		 The proof of the Lemma \ref{le res estimates} is finished.
	\end{proof}
	
	Finally, we show that $\{\tilde{e}_{k}^{(o)}, k\geq 1\}$ forms a complete orthonormal basis for $\mathcal{H}$, a fact that will play a key role in the linear stability analysis in Section \ref{section COSSP}.
	\begin{lemma}\label{le cobasis}
		$\{\tilde{e}_{k}^{(o)}, k\geq 1\}$ is a complete orthonormal basis for $\mathcal{H}$.
	\end{lemma}
	\begin{proof}[Proof of Lemma \ref{le cobasis}]
		It holds that
		\begin{align*}
			\frac{\partial_{x}\tilde{e}_{k}^{(o)}}{\sin (\frac{x}{2})} = -2\sin((k+\frac{1}{2})x),
		\end{align*}
		for any \(k\geq 1\), which leads to
		\begin{align*}
			\big\langle\tilde{e}_{k}^{(o)},\tilde{e}_{l}^{(o)}\big\rangle_{\mathcal{H}} = \delta_{kl},
		\end{align*}
		for \(k,l\geq 1\). If $\xi \in \mathcal{H}$ satisfies
		\begin{align*}
			\big\langle\xi,\tilde{e}_{k}^{(o)}\big\rangle_{\mathcal{H}} = 0,
		\end{align*}
		for \(k\geq 1\), then
		\begin{align*}
			\int_{-\pi}^{\pi} \frac{\partial_{x}\xi}{\sin (\frac{x}{2})} \sin((k+\frac{1}{2})x) \, \mathrm{d}x = 0,
		\end{align*}
		for \(k\geq 1\). The above equality holds for $k=0$ as well, due to the fact that
		\begin{align*}
			\int_{-\pi}^{\pi}\partial_{x}\xi \, \mathrm{d}x = 0.
		\end{align*}
		Since $\{\sin((k+\frac{1}{2})x), k\geq 0\}$ is a complete basis for odd functions in $L^2(\mathbb{T})$, it follows that $\partial_{x} \xi = 0$, which implies $\xi = 0$.
	\end{proof}

	\section{Blow-up for the Inviscid Generalized Proudman-Johnson Equation}\label{section COSSP}
	In this section, we give the proof of Theorem \ref{them profile}. To establish singularity formation to \ref{PJE} with $\nu=0$, we make use of the dynamical rescaling formulation, which connects a self-similar singularity to the steady state of the rescaled system.
		
	    Let $(\omega(x,t),u(x,t))$ denote the solution to \ref{PJE}. It is straightforward to verify that
	\begin{align}\label{tilde ompsi}
		\tilde{\omega}(x,\tau) = C_{\omega}(\tau)\omega(x,t(\tau)), \quad \tilde{u}(x,\tau) = C_{\omega}(\tau)u(x,t(\tau))
	\end{align}
	are the solutions to the dynamic rescaling equations
	\begin{align}\label{dre}
		\tilde{\omega}_{\tau}(x,\tau) + \tilde{u}\tilde{\omega}_{x}(x,\tau) = c_{\omega}(\tau)\tilde{\omega} + a\tilde{u}_{x}\tilde{\omega}, \quad \tilde{u}_{xx} = \tilde{\omega}, \quad  \tau >0, \ x\in\mathbb{T},
	\end{align}
	where
	\begin{align}
		C_{\omega}(\tau) = \exp\left(\int_{0}^{\tau} c_{\omega}(s)ds\right), \quad t(\tau) = \int_{0}^{\tau} C_{\omega}(s)ds.
	\end{align}
	Here, we do not rescale the spatial variable $x,$ as we are interested in a blow-up solution that is neither focusing nor expanding over a fixed spatial period. Consequently, the scaling factors for $\omega,u$ are the same.
	
	If $c_{\omega}(\tau)\leq -c<0$ for some constant $c > 0$ and for all $\tau>0$, the solution $\tilde{\omega}$ is nontrivial, for instance $\|\tilde{\omega}(\tau,\cdot)\|_{L^{\infty}} \geq C >0$ for all $\tau > 0$, then we obtain
	\begin{align*}
		C_{\omega}(\tau) \leq e^{-c\tau}, \quad t(\infty) \leq \int_{0}^{\infty}e^{-c\tau}d\tau = c^{-1} < +\infty
	\end{align*}
	and that
	\begin{align*}
		|\omega(x,t(\tau))| = C_{\omega}(\tau)^{-1}|\tilde{\omega}(x,\tau)|,
	\end{align*}
	blows up at finite time $T = t(\infty)$.
		
	On the other hand, if $c_{\omega}(\tau) \geq c>0$ for some constant $c>0$ and $\tilde{\omega}(x,\tau)$ remains bounded, for example, $\|\tilde{\omega}(\tau,\cdot)\|_{L^{\infty}} \leq C $ for some constant $C>0$, then
	\begin{align*}
		C_{\omega}(\tau) \geq e^{c\tau}, \quad t(\tau) \geq \int_{0}^{\tau}e^{cs}ds.
	\end{align*}
	Consequently,
	\begin{align*}
		|\omega(x,t(\tau))| = C_{\omega}(\tau)^{-1}|\tilde{\omega}(x,\tau)| \leq e^{-c\tau}|\tilde{\omega}(x,\tau)| \leq e^{-c\tau}C,
	\end{align*}
	decays for large $\tau$. Since $t(\tau)\rightarrow \infty$ as $\tau\rightarrow\infty$, the Beale-Kato-Majda criterion implies global existence of the solution to \ref{PJE}.
	
		If $(\tilde{\omega},c_{\omega}(\tau))\to(\omega_{\infty},c_{\omega,\infty})$ as $\tau\to\infty$, where $(\omega_{\infty},c_{\omega,\infty})$ is a steady state of the dynamically rescaled equation \ref{dre}, then one may verify that
	\begin{align}
		\omega(x,t) = \frac{1}{1+c_{\omega,\infty}t}\omega_{\infty}(x),
	\end{align}
	is a self-similar solution to the original equation \ref{PJE}. In view of this correspondence, we do not distinguish the steady state of the dynamic rescaling equation \eqref{dre} from the self-similar profile of the original equation \eqref{PJE}.
	
	Now we give the proof of Theorem \ref{them profile}.
	\begin{proof}[Proof of Theorem \ref{them profile}]

	\textbf{Step 1. Perturbation Equations.}\label{section nonlinear}
	
	In this step, we employ the steady state corresponding to the case \(a = 1\) to construct an approximate steady state for \eqref{dre}, which is
	\begin{align}\label{ap steady state}
		\bar{\omega} (x) = -\sin(x), \quad \bar{u} (x) = \sin(x), \quad \bar{c}_{\omega} (\tau) = (1-a)\bar{u}_{x}(0) = 1-a.
	\end{align}
	We consider odd perturbations of the form
	\begin{align*}
		\tilde{\omega} (x,\tau) = \bar{\omega} (x) + \hat{\omega} (x,\tau),\, \tilde{u} (x,\tau) =  \bar{u} (x) + \hat{u} (x,\tau),\, c_{\omega}(\tau) =  \bar{c}_{\omega} (\tau) + \hat{c}_{\omega} (\tau), \quad  \tau>0, \ x\in\mathbb{T}
	\end{align*}
	and impose the following normalization condition on \(\hat{c}_{\omega}(\tau)\) in \eqref{dre}.
	\begin{align}\label{hat comega}
		\hat{c}_{\omega}(\tau) = (1-a) \hat{u}_{x}(0,\tau).
	\end{align}
	This normalization ensures the conservation of the quantity \(\bar{\omega}_{x}(0) + \hat{\omega}_{x}(0,\tau)\) over time. Indeed, we have
	\begin{align*}
		\frac{\mathrm{d}}{\mathrm{d}\tau} (\hat{\omega}_{x}(0,t) + \bar{\omega}_{x}(0)) = & a \big(\hat{u}_{x}(0,\tau) + \bar{u}_{x}(0)\big) \big(\hat{\omega}_{x}(0,\tau) + \hat{\omega}_{x}(0,\tau)\big)  - \big(\hat{u}_{x}(0,\tau) + \bar{u}_{x}(0)\big) \big(\hat{\omega}_{x}(0,\tau) + \bar{\omega}_{x}(0)\big) \\
		& + \big(\hat{c}_{\omega} + \bar{c}_{\omega}\big) \big(\hat{\omega}_{x}(0,\tau) + \hat{\omega}_{x}(0,\tau)\big) =0.
	\end{align*}
	Then, the perturbation equation takes the following form
	\begin{equation}\label{perte}
		\begin{aligned}
			\hat{\omega}_{\tau}  = & -\bar{u} \hat{\omega}_{x} + a\bar{u}_{x}\hat{\omega} + a\bar{\omega}\hat{u}_{x} - \bar{\omega}_{x} \hat{u} + \bar{c}_{\omega} \hat{\omega} + \bar{\omega} \hat{c}_{\omega} + N(\hat{\omega}) + F(\bar{\omega})\\
		 :=	& \mathcal{L}_{a} \hat{\omega} + N(\hat{\omega}) + F(\bar{\omega}),
		\end{aligned}
	\end{equation}
	where \(\mathcal{L}_{a}\) denotes the linearized operator, \(N(\hat{\omega})\) and \(F(\bar{\omega})\) represent the nonlinear and error  terms, respectively, which are
	\begin{align}\label{non term}
		\begin{aligned}
			N(\hat{\omega}) & = a \hat{u}_{x} \hat{\omega} - \hat{u} \hat{\omega}_{x} + \hat{c}_{\omega}\hat{\omega},\\
			 F(\bar{\omega}) & = (a\bar{u}_{x} + \bar{c}_{\omega})\bar{\omega} - \bar{u} \bar{\omega}_{x} = (1-a)\sin(x) \big(\cos(x) - 1\big).
		\end{aligned}
	\end{align}
	Plugging the approximate steady state \ref{ap steady state} and the normalization condition into \ref{perte} yields
	\begin{equation}\label{definition L}
		\begin{aligned}
			\mathcal{L}_{1}\hat{\omega} = & -\sin(x) \hat{\omega}_{x} + \cos(x) \hat{\omega} - \sin(x) \hat{u}_{x} + \cos(x) \hat{u},\\
			\mathcal{L}_{a} \hat{\omega} = & -\sin(x) \hat{\omega}_{x} + (a\cos(x) + 1 - a)\hat{\omega} \\
			& - (a\hat{u}_{x} + (1-a)\hat{u}_{x}(0,\tau))\sin(x) + \cos(x) \hat{u}\\
			= & \mathcal{L}_{1}\hat{\omega} + (1-a)\big(- \hat{u}_{x}(0,\tau)\sin(x) + \hat{\omega} - \cos(x) \hat{\omega} + \sin(x) \hat{u}_{x}\big)\\
			:= & \mathcal{L}_{1} \hat{\omega} + (1-a)\mathcal{A}\hat{\omega}.
		\end{aligned}
	\end{equation}
	To study the stability of the dynamically rescaled equation and its convergence to a steady state, we perform a weighted $\dot{H}^1$ estimate using a singular weight $\rho$, with the corresponding weighted norm defined by
	\begin{align}
		\rho = \frac{1}{4\pi \sin^2\left(\frac{x}{2}\right)}, \quad \|f\|_{\mathcal{H}}^{2} = \left\langle f_x^2, \rho \right\rangle = \int_{\mathbb{T}} f_x^{2} \rho\, \mathrm{d}x.
	\end{align}
	For initial perturbation satisfying $\hat{u}_{x}(0,0)=0$, it follows that $u_{x}(0,\tau)=0$ for all $\tau>0$, ensuring that $\|\hat{u}\|_{\mathcal{H}}$ is well-defined.
	
	\textbf{Step 2. Nonlinear stability analysis.}\label{subsection nonlinear}
	
	The dominant part $\mathcal{L}_1 \hat{\omega}$ of the linearized operator provides a damping effect. We state a lemma that will be used subsequently, of which proof is given in the Appendix.
	\begin{lemma}\label{le operator}
		It holds that
		\begin{align}\label{eq:estimate-L1}
			\left\langle\mathcal{L}_{1} \hat{\omega},\hat{\omega} \right\rangle_{\mathcal{H}} \leq -\frac{1}{2}\|\hat{\omega}\|_{\mathcal{H}}^{2}.
		\end{align}
	\end{lemma}
	For the term $\left\langle \mathcal{A}\hat{\omega}, \hat{\omega} \right\rangle_{\mathcal{H}}$, invoking Lemmas \ref{le sinfx} and \ref{le estimates}, we obtain
	\begin{align}
		\begin{aligned}\label{eq:estimate-Aw}
			\left\langle \mathcal{A}\hat{\omega}, \hat{\omega} \right\rangle_{\mathcal{H}} & = \left\langle \sin(x) \hat{u}_{xx},\hat{\omega}_{x} \rho\right\rangle + \left\langle \cos(x)(\hat{u}_{x} -\hat{u}_{x}(0,\tau)), \hat{\omega}_{x} \rho \right\rangle + \left\langle \sin(x) \hat{\omega}, \hat{\omega}_{x} \rho \right\rangle\\
			& \quad + \left\langle(1-\cos(x)),\hat{\omega}_{x}^{2} \rho \right\rangle\\
			&= 2\left\langle \sin(x) \hat{\omega}, \hat{\omega}_{x} \rho \right\rangle + \left\langle \cos(x) (\hat{u}_{x} - \hat{u}_{x}(0,\tau)), \hat{\omega}_{x} \rho \right\rangle + \left\langle(1-\cos(x)), \hat{\omega}_{x}^{2} \rho \right\rangle\\
			&\lesssim \left(\|\hat{\omega}\rho^{1/2}\|_{L^{\infty}} + \|(\hat{u}_{x} - \hat{u}_{x}(0,\tau)) \rho^{1/2}\|_{L^{2}} + \|\hat\omega\|_{\mathcal{H}} \right) \|\hat\omega\|_{\mathcal{H}}\\
			& \lesssim \|\hat{\omega}\|_{\mathcal{H}}^{2}.
		\end{aligned}
	\end{align}
	Hence,
	\begin{align}\label{eq:linear estimate}
		\begin{aligned}
			\left\langle \mathcal{L}_{a}\hat{\omega}, \hat{\omega} \right\rangle_{\mathcal{H}} & = \left\langle \mathcal{L}_{1}\hat{\omega}, \hat{\omega} \right\rangle_{\mathcal{H}} + (1-a)\left\langle \mathcal{A}\hat{\omega}, \hat{\omega} \right\rangle_{\mathcal{H}}\\
			& \leq - (\frac{1}{2} - C|1-a|) \|\hat{\omega} (\tau)\|_{\mathcal{H}}^{2}.
		\end{aligned}
	\end{align}

	Applying \ref{eq:linear estimate}, we obtain
	\begin{align}
		\frac{1}{2} \frac{\mathrm{d}}{\mathrm{d} \tau} \|\hat{\omega} (\tau) \|_{\mathcal{H}}^{2} \leq - (\frac{1}{2} - C|1-a|) \|\hat{\omega} (\tau)\|_{\mathcal{H}}^{2} + \left\langle N(\hat{\omega}), \hat{\omega}\right\rangle_{\mathcal{H}} + \left\langle F(\bar{\omega}), \hat{\omega}\right\rangle_{\mathcal{H}},
	\end{align}
	where \(\mathcal{A}\hat{\omega}, N(\hat{\omega})\) and \(F(\bar{\omega})\) are defined as in \ref{non term} and \ref{definition L} respectively.
	We now proceed with the stability analysis.
	For the nonlinear term, we apply Lemma \ref{le estimates} to obtain
	\begin{align}\label{eq:estimate-Nw}
		\begin{aligned}
			\left\langle N(\hat{\omega}), \hat{\omega}\right\rangle_{\mathcal{H}} & = a \left\langle
			\hat{\omega}^{2}, \hat{\omega}_{x} \rho \right\rangle + (a-1) \left\langle \hat{u}_{x}, \hat{\omega}_{x}^{2} \rho \right\rangle + (1-a) \left\langle \hat{u}_{x}(0,\tau), \hat{\omega}_{x}^{2} \rho \right\rangle  - \left\langle \hat{u} \hat{\omega}_{xx}, \hat{\omega}_{x} \rho \right\rangle \\
			& \lesssim \left(\|\hat{\omega}\rho^{1/2}\|_{L^{\infty}} + \| \hat{u}_{x}\|_{L^{\infty}}\right) \|\hat{\omega}\|_{\mathcal{H}}^{2} - \left\langle \hat{u} \hat{\omega}_{xx}, \hat{\omega}_{x} \rho \right\rangle\\
			& \lesssim \|\hat{\omega}\|_{\mathcal{H}}^{3}.
		\end{aligned}
	\end{align}
	For the last term, we have employed integration by parts together with Lemma \ref{le estimates} to derive
	\begin{align*}
		\left\langle \hat{u} \hat{\omega}_{xx}, \hat{\omega}_{x} \rho \right\rangle & =  - \frac{1}{2} \left\langle \hat{u}_{x}, \hat{\omega}_{x}^{2} \rho \right\rangle + \frac{1}{2} \big\langle \frac{\cos(\frac{x}{2}) \hat{u}}{\sin(\frac{x}{2})}, (\hat{\omega}_{x})^{2} \rho \big\rangle\\
		& \lesssim  \left(\|\hat{u}_{x}\|_{L^{\infty}} + \bigg\|\frac{\hat{u}}{\sin(\frac{x}{2})}\bigg\|_{L^{\infty}}\right) \|\hat{\omega}\|_{\mathcal{H}}^{2}\\
		& \lesssim  \|\hat{\omega}\|_{\mathcal{H}}^{3}.
	\end{align*}
	For the error term, we deduce from Lemma \ref{le estimates} that
	\begin{align}\label{eq:estimate-Fw}
		\begin{aligned}
			\left\langle F(\bar{\omega}), \hat{\omega} \right\rangle_{\mathcal{H}} & = (1-a) \left(\left\langle \cos(x)(\cos(x) -1 ), \hat{\omega}_{x} \rho \right\rangle - \left\langle\sin^{2}(x), \hat{\omega}_{x} \rho \right\rangle \right)\\
			& \lesssim  |1-a| \|\hat{\omega}\|_{\mathcal{H}}.
		\end{aligned}
	\end{align}
	
	\textbf{Step 3. Finite time blow-up analysis.}\label{subsection blowup}
	
	Combining the above estimates \ref{eq:estimate-Aw}, \ref{eq:estimate-Nw} and \ref{eq:estimate-Fw}, we arrive at
	\begin{align}
		\frac{\mathrm{d}}{\mathrm{d} \tau} \|\hat{\omega} (\tau) \|_{\mathcal{H}} \leq - (\frac{1}{2} - C|1-a|) \|\hat{\omega} (\tau)\|_{\mathcal{H}} + C\|\hat{\omega} (\tau)\|_{\mathcal{H}}^{2} + C|1-a|.
	\end{align}
	By employing a standard bootstrap argument, one can show the existence of absolute constants  $\delta,C>0$ such that if $|1-a|<\delta$ and $\|\hat{\omega} (0)\|_{\mathcal{H}} \leq |1-a|$, then we have $\|\hat{\omega} (\tau)\|_{\mathcal{H}} \leq C|1-a|$ for all time. In particular, we have
	\begin{align*}
		|\hat{u}_{x} (0) | \leq \|\hat{u}_{x}\|_{L^{\infty}} \lesssim \|\hat{\omega}\|_{\mathcal{H}} \leq C|1-a|.
	\end{align*}
	It follows that
	\begin{align*}
		c_{\omega} (\tau)  = \hat{c}_{\omega} (\tau) + \bar{c}_{\omega}  = (1-a) \left(\hat{u}_{x}(0,\tau) + \bar{u}_{x}(0)\right) = (1-a)(1 + \hat{u}_{x}(0,\tau)) < 0
	\end{align*}
	 for \(a \in (1, 1+\delta)\).
	
	 Consequently, the solution to \ref{PJE} blows up in finite time for $a\in (1,1+\delta)$, completing the proof of part (1) of Theorem \ref{them profile}.
	
	\textbf{Step 4. Convergence to self-similar profiles.}\label{step3:convergence}
	
We now turn to the proof of part (2) of Theorem \ref{them profile}.	To establish the convergence of the solution to a steady state, we need to estimate weighted norms of $\hat{\omega}_{\tau} $. As pointed out in \cite{CHH2021}, obtaining stability estimates in higher-order Sobolev norms is crucial for closing the estimate. We utilize the weighted derivative $D_{x} = \sin(x) \partial_{x}$, which has been similarly employed in \cite{CH2021,E2021,EG2021,C2021} for stability analysis.
	
	In order to get the estimate for $\|\hat{\omega}_{\tau}\|_{\mathcal{H}}$, we first establish the estimate for $\|\hat{\omega}\|_{\mathcal{W}}$. Recalling the definition of $\mathcal{L}_{a}$ in \ref{definition L}, we obtain
	\begin{align*}
		\mathcal{L}_{a}  = & -\sin(x) \hat{\omega}_{x} + (a\cos(x) + 1 -a) \hat{\omega} - \left(a \hat{u}_{x} + (1-a) \hat{u}_{x}(0,\tau)\right) \sin(x) + \cos xu\\
		= & (-D_{x}\hat{\omega} + \hat{\omega} \cos(x)) + (1-a)(1-\cos(x)) \hat{\omega} + a (-\sin(x) \hat{u}_{x} + \cos(x) \hat{u}) \\
		& + (1-a) \left(\hat{u}\cos(x) - \hat{u}_{x}(0,\tau) \sin(x) \right)\\
		:= & \mathcal{T}_{1} + \mathcal{T}_{2} + \mathcal{T}_{3} + \mathcal{T}_{4}.
	\end{align*}
	 The weighted derivative $D_{x}$ satisfies the Leibniz rule
	\begin{align*}
		D_{x}(fg) = gD_{x}f + fD_{x}g.
	\end{align*}
	Then, direct calculations give
	\begin{align*}
		\partial_{x}D_{x}\mathcal{T}_{1} = & \partial_{x}(-\sin (x) \partial_{x}D_{x}\omega + \cos(x) D_{x}\omega + \omega D_{x}\cos(x))\\
		= & -\sin(x) \partial_{x}\partial_{x}D_{x}\omega - \cos(x) \partial_{x}D_{x}\omega - \sin(x) D_{x}\omega + \cos(x) \partial_{x}D_{x}\omega\\
		& + \partial_{x}\omega D_{x}\cos(x) + \omega\partial_{x}D_{x}\cos(x)\\
		= & -\sin(x) \partial_{x}\partial_{x}D_{x}\omega - \sin^{2}(x) \partial_{x}\omega - \sin^{2}(x) \partial_{x}\omega - 2\sin(x) \cos(x) \omega,
	\end{align*}
	
	\begin{align*}
		\partial_{x} D_{x}\mathcal{T}_{2} = & (1-a) \partial_{x}\left((1-\cos(x)) D_{x}\hat{\omega} + \hat{\omega} D_{x}(1-\cos(x)) \right)\\
		= & (1-a)\partial_{x} \left((1-\cos(x)) D_{x}\hat{\omega} + \hat{\omega} \sin^{2}(x) \right),
	\end{align*}
	
	\begin{align*}
		\partial_{x} D_{x} \mathcal{T}_{3} = & a \partial_{x}(-\sin(x) D_{x}\hat{u}_{x} - \hat{u}_{x} D_{x}\sin(x) + \cos(x) D_{x}\hat{u} + \hat{u} D_{x}\cos(x))\\
		= & a \partial_{x} (-\sin^{2}(x) u_{xx} -\sin(x) \cos(x) \hat{u}_{x} + \sin(x) \cos(x) \hat{u}_{x} -\sin^{2}(x) \hat{u})\\
		= & a \partial_{x}(-\sin^{2}(x) \hat{\omega} -\sin^{2}(x) \hat{u})
	\end{align*}
	and
	\begin{align*}
		\partial_{x} D_{x} \mathcal{T}_{4} = & (1-a) \partial_{x}\left(\sin(x)\cos (x)\hat{u}_{x} - \sin^{2}(x) \hat{u} - \hat{u}_{x}(0,\tau) \sin(x)\cos(x) \right)\\
		= & (1-a)\partial_{x} \left(\sin (x) \cos(x) (\hat{u}_{x} - \hat{u}_{x}(0,\tau)) - \sin^{2}(x) \hat{u}\right)\\
		= &  (1-a)(\hat{u}_{x} - \hat{u}_{x}(0,\tau)) \partial_{x}(\sin(x)\cos(x)) + (1-a) \sin(x)\cos(x) \hat{\omega} - (1-a)\partial_{x}(\sin^{2}(x)\hat{u}).
	\end{align*}
	Therefore, we obtain
		\begin{align*}
			\left\langle \partial_{x}D_{x} \mathcal{L}_{a}, \partial_{x}D_{x} \hat{\omega} \rho\right\rangle = & \left\langle \partial_{x}D_{x}\mathcal{T}_{1}, \partial_{x}D_{x} \hat{\omega} \rho\right\rangle + \left\langle\partial_{x}D_{x}\mathcal{T}_{2},\partial_{x}D_{x} \hat{\omega} \rho\right\rangle \\
			& + \left\langle \partial_{x}D_{x}\mathcal{T}_{3},\partial_{x}D_{x} \hat{\omega} \rho\right\rangle +  \left\langle \partial_{x}D_{x}\mathcal{T}_{4}, \partial_{x}D_{x} \hat{\omega}\rho\right\rangle\\
			= & \left\langle -\sin(x) \partial_{x}\partial_{x}D_{x} \hat{\omega} + (1-a)(1-\cos(x)) \partial_{x}D_{x}\hat{\omega} + \mathrm{l.o.t.}, \partial_{x}D_{x} \hat{\omega} \rho \right\rangle\\
			\leq & (-\frac{1}{2} + |1-a|) \|\partial_{x}D_{x} \hat{\omega} \rho^{1/2}\|_{L^{2}}^{2} + C\|\hat{\omega}\|_{\mathcal{H}} \|\partial_{x}D_{x} \hat{\omega}\rho^{1/2}\|_{L^{2}}\\
			\leq & (-\frac{1}{4} + |1-a|) \|\partial_{x}D_{x} \hat{\omega} \rho^{1/2}\|_{L^{2}}^{2} + C\|\hat{\omega}\|_{\mathcal{H}}^{2}.
		\end{align*}
	The nonlinear term is estimated as follows.
	\begin{align*}
		\partial_{x}D_{x} N(\hat{\omega}) = & a\partial_{x}(\sin(x) \hat{u}_{xx} \hat{\omega} + \sin(x) \hat{u}_{x} \hat{\omega}_{x}) - \partial_{x}(\sin(x) \hat{u}_{x}\hat{\omega}_{x} + \sin(x) \hat{u} \hat{\omega}_{xx})\\
		& + (1-a)\hat{u}_{x}(0,\tau) \partial_{x}D_{x} \hat{\omega}\\
		= & a\partial_{x}(\sin(x) (\hat{\omega})^{2}) + (a-1)\partial_{x}(\sin(x) \hat{u}_{x} \hat{\omega}_{x}) - \partial_{x}(\sin(x) \hat{u} \hat{\omega}_{xx})\\
		& + (1-a)\hat{u}_{x}(0,\tau) \partial_{x}D_{x}\hat{\omega}\\
		:= & N_{1} + N_{2} + N_{3} + N_{4}.
	\end{align*}
	Using H\"{o}ler's inequality, Lemma \ref{le estimates} and Poincar\'{e}'s inequality, we obtain
	\begin{align*}
		\left\langle N_{1}, \partial_{x}D_{x} \hat{\omega}\rho\right\rangle = & \left\langle \cos(x) \omega^{2}, \partial_{x}D_{x} \hat{\omega} \rho\right\rangle + \left\langle 2\sin(x) \hat{\omega} \hat{\omega}_{x}, \partial_{x}D_{x} \hat{\omega} \rho \right\rangle\\
		\lesssim & \|\hat{\omega} \rho^{1/2}\|_{L^{\infty}} \|\hat{\omega}\|_{L^{2}} \|\partial_{x}D_{x} \hat{\omega} \rho^{1/2}\|_{L^{2}}\\
		\lesssim & \|\hat{\omega}\|_{\mathcal{W}}^{3}.
	\end{align*}
	Similarly, we have
	\begin{align*}
		\left\langle N_{2} ,\partial_{x}D_{x} \hat{\omega} \rho\right\rangle = & \left\langle \hat{u}_{xx} D_{x} \hat{\omega}, \partial_{x}D_{x} \hat{\omega} \rho \right\rangle + \left\langle \hat{u}_{x} \partial_{x}D_{x}\hat{\omega}, \partial_{x}D_{x}\hat{\omega} \rho\right\rangle \\
		\lesssim & \|\hat{\omega}\|_{L^{\infty}} \|\hat{\omega}\|_{\mathcal{H}} \|\partial_{x}D_{x}\hat{\omega} \rho^{1/2}\|_{L^{2}} + \|\hat{u}_{x}\|_{L^{\infty}} \|\partial_{x}D_{x} \hat{\omega} \rho\|_{L^{2}}^{2}\\
		\lesssim & \|\hat{\omega}\|_{\mathcal{W}}^{3}
	\end{align*}
	and
	\begin{align*}
		\left\langle N_{4} , \partial_{x}D_{x} \hat{\omega} \rho\right\rangle = & \left\langle (1-a) \hat{u}_{x}(0,\tau) \partial_{x}D_{x} \hat{\omega},  \partial_{x}D_{x} \hat{\omega} \rho\right\rangle\\ \lesssim & \|\hat{u}_{x}\|_{L^{\infty}} \|\partial_{x}D_{x} \hat{\omega}\rho^{1/2}\|_{L^{2}}\\
		\lesssim & \|\hat{\omega}\|_{\mathcal{W}}^{3}.
	\end{align*}
	Next we estimate $\left\langle N_{3} ,\partial_{x}D_{x} \hat{\omega}\rho\right\rangle$. Since
	\begin{align*}
		N_{3} = & - \partial_{x}(\sin(x) \hat{u} \hat{\omega}_{xx}) = -\partial_{x}(\hat{u} D_{x} \hat{\omega}_{x})\\
		= & -\partial_{x}(\hat{u} \partial_{x}D_{x} \hat{\omega} - \hat{u}\cos(x) \hat{\omega}_{x}).
	\end{align*}
	Applying integration by parts and Lemma \ref{le estimates}, we obtain
	\begin{align*}
		-\left\langle \hat{u} \partial_{x}\partial_{x}D_{x} \hat{\omega}, \partial_{x}D_{x} \hat{\omega}\rho\right\rangle & = \left\langle \hat{u}_{x}, (\partial_{x}D_{x}\hat{\omega})^{2}\rho\right\rangle - \big\langle\frac{\hat{u}\cos(\frac{x}{2})}{\sin(\frac{x}{2})},  (\partial_{x}D_{x} \hat{\omega})^{2} \rho \big\rangle\\
		& \leq \big(\|\hat{u}_{x}\|_{L^{\infty}} + \bigg\|\frac{\hat{u}}{\sin(\frac{x}{2})}\bigg\|_{L^{\infty}}\big) \|\partial_{x}D_{x} \hat{\omega}\rho^{1/2}\|_{L^{2}}^{2}\\
		& \lesssim \|\hat{\omega}\|_{\mathcal{W}}^{3}.
	\end{align*}
	It follows that
	\begin{align*}
		\left\langle N_{3} , \partial_{x}D_{x} \hat{\omega} \rho \right\rangle = & - \left\langle \hat{u} \partial_{x}\partial_{x}D_{x} \hat{\omega},  \partial_{x}D_{x} \hat{\omega} \rho \right\rangle - \left\langle \partial_{x} \hat{u} \partial_{x}D_{x} \hat{\omega}, \partial_{x}D_{x} \hat{\omega} \rho\right\rangle\\
		& + \left\langle\partial_{x}(\hat{u} \cos(x) \hat{\omega}_{x}), \partial_{x}D_{x} \hat{\omega} \rho\right\rangle\\
		\lesssim & \|\hat{\omega}\|_{\mathcal{W}}^{3}.
	\end{align*}
	Combining the estimates for $N_{i} (i=1,2,3,4)$, we obtain
	\begin{align}\label{nonlinear estimate}
		\left\langle \partial_{x} D_{x} N(\hat{\omega}), \partial_{x}D_{x} \hat{\omega} \rho \right\rangle \leq C\|\hat{\omega}\|_{\mathcal{W}}^{3}.
	\end{align}
	For the estimate of error term, since
	\begin{align*}
		F(\bar{\omega}) = & (a\bar{u}_{x} + \bar{c}_{\omega})\bar{\omega} - \bar{u}\bar{\omega}_{x}\\
		= & -a\cos(x) \sin(x) - (1-a)\sin(x) + \sin(x) \cos(x)\\
		= & -(1-a) \sin(x)(1 - \cos(x)) = -2(1-a)\sin(x) \sin^{2}(\frac{x}{2}),
	\end{align*}
	it follows that
	\begin{align*}
		\left\langle\partial_{x}D_{x}F(\bar{\omega}), \partial_{x}D_{x} \hat{\omega} \rho\right\rangle \lesssim |1-a| \|\hat{\omega}\|_{\mathcal{W}}.
	\end{align*}
	Define
	\begin{align*}
		E(\tau)^{2} = \|\partial_{x}D_{x} \hat{\omega}\rho^{1/2}\|_{L^{2}}^{2} + \mu \|\hat{\omega}\|_{\mathcal{H}}^{2}
	\end{align*}
	for some absolute constant $\mu > 1$. Then
	\begin{align*}
		\frac{1}{2}\frac{\mathrm{d}}{\mathrm{d} \tau} E(\tau)^{2} \leq -(\frac{1}{8} - C|1-a|) E(\tau)^{2} + C\left(\|\hat{\omega}\|_{\mathcal{H}}^{3} + \|\hat{\omega}\|_{\mathcal{W}}^{3}\right) + C|1-a| \left(\|\hat{\omega}\|_{\mathcal{H}} + \|\hat{\omega}\|_{\mathcal{W}}\right).
	\end{align*}
	Noting that
	\begin{align*}
		|2\sqrt{\pi} \rho^{1/2} \partial_{x}D_{x} \hat{\omega}|  = \bigg|\frac{1}{\sin(\frac{x}{2})}(\cos(x) \hat{\omega}_{x} + \sin(x) \hat{\omega}_{xx})\bigg|,
	\end{align*}
	we obtain the equivalence of norms
	\begin{align}\label{equiv norm}
		\|\hat{\omega}\|_{\mathcal{W}}^{2} \lesssim \|\hat{\omega}\|_{\mathcal{H}}^{2} + \|\partial_{x} D_{x}\hat{\omega} \rho^{1/2}\|_{L^{2}}^{2} \lesssim \|\hat{\omega}\|_{\mathcal{W}}^{2}.
	\end{align}
	Then,
	\begin{align*}
		\|\hat{\omega}\|_{\mathcal{W}}^{2} \lesssim E(\tau)^{2} \lesssim \|\hat{\omega}\|_{\mathcal{W}}^{2},
	\end{align*}
	which implies that
	\begin{align*}
		\frac{1}{2} \frac{\mathrm{d}}{\mathrm{d} \tau} E(\tau)^{2} \leq -(\frac{1}{8} - C|1-a|) E(\tau)^{2} + CE(\tau)^{3} + C|1-a|E(\tau).
	\end{align*}
	Hence, there exist a sufficiently small constant $\delta_{1}<\delta$ and a constant $C>0$ such that, if $|1-a| < \delta_{1}$ and $E(0) < \delta_{1}$, then $E(\tau) \leq C\delta_{1}$ for all $t\geq 0$. This can be established via a standard bootstrap argument. In view of \ref{equiv norm}, we obtain
	\begin{align}\label{omega W}
		\|\hat{\omega}\|_{\mathcal{W}} \leq C E(\tau)\leq C\delta_1.
	\end{align}

	We next show that $\|\hat{\omega}_{\tau}\|_{\mathcal{H}}$ decays exponentially as $\tau \to +\infty$. Direct estimate gives
	\begin{align*}
		\frac{1}{2} \frac{\mathrm{d}}{\mathrm{d} \tau} \left\langle \hat{\omega}_{\tau}, \hat{\omega}_{\tau} \right\rangle_{\mathcal{H}} = \left\langle (\mathcal{L}_{a} \hat{\omega})_{\tau}, \hat{\omega}_{\tau} \right\rangle_{\mathcal{H}} + \left\langle (N(\hat{\omega}))_\tau, \hat{\omega}_{\tau} \right\rangle_{\mathcal{H}}.
	\end{align*}
	We obtain a damping effect arising from the linear terms, while the small error terms associated with $\bar{\omega}$ and $\bar{u}$ vanish.As a result, we derive
	\begin{align*}
		\frac{1}{2} \frac{\mathrm{d}}{\mathrm{d} \tau}  \|\hat{\omega}_{\tau}\|_{\mathcal{H}}^{2} \leq - (\frac{1}{2} - |1-a|) \|\hat{\omega}_{\tau}\|_{\mathcal{H}}^{2} + \left\langle (N(\hat{\omega}))_{\tau}, \hat{\omega}_{\tau}\right\rangle_{\mathcal{H}}.
	\end{align*}
	Using similar estimates as in \textbf{Step 2}, we obtain
	\begin{align*}
		|\left\langle (N(\hat{\omega}))_{\tau}, \hat{\omega}_{\tau} \right\rangle_{\mathcal{H}}| \lesssim E(\tau) \|\hat{\omega}_{\tau}\|_{\mathcal{H}}^{2} + |\left\langle \hat{u}_{\tau} \hat{\omega}_{xx}, \hat{\omega}_{\tau,x} \rho\right\rangle| +|\left\langle \hat{u} \hat{\omega}_{\tau,xx}, \hat{\omega}_{\tau,x} \rho \right\rangle|.
	\end{align*}
	For the term $\left\langle \hat{u}_{\tau} \hat{\omega}_{xx}, \hat{\omega}_{\tau,x} \rho \right\rangle $, we have
	\begin{align*}
		\left\langle \hat{u}_{\tau} \hat{\omega}_{xx}, \hat{\omega}_{\tau,x} \rho \right\rangle = & \left\langle \hat{u}_{\tau} \rho^{1/2} \hat{\omega}_{xx}, \hat{\omega}_{\tau,x} \rho^{1/2} \right\rangle\\
		= & \frac{1}{4\pi} \big\langle \frac{\hat{u}_{\tau}}{\sin(\frac{x}{2})} \hat{\omega}_{xx}, \frac{\hat{\omega}_{\tau,x}}{\sin(\frac{x}{2})}\big\rangle\\
		= & \frac{1}{4\pi} \big\langle \frac{2 \hat{u}_{\tau}}{\sin(x)} \hat{\omega}_{xx} \cos(\frac{x}{2}), \frac{\hat{\omega}_{\tau,x}}{\sin(\frac{x}{2})} \big\rangle\\
		\lesssim & \bigg\|\frac{\hat{u}_{\tau}}{\sin(x)} \bigg\|_{L^{\infty}} E(\tau) \|\hat{\omega}_{\tau}\|_{\mathcal{H}}\\
		\lesssim & E(\tau) \|\hat{\omega}_{\tau}\|_{\mathcal{H}}^{2}.
	\end{align*}
	To estimate the term $\bigg\|\frac{\hat{u}_{t}}{\sin(x)}\bigg\|_{L^{\infty}}$, we exploit the oddness and periodicity of $\hat{u}_{\tau}$ in the spatial variable \(x\).  In particular, since $\hat{u} (\pi,\tau) = \hat{u} (0,\tau) = 0$,  it suffices to estimate the norm on $[0,\pi]$. Observing that $\sin(x) \leq \frac{2}{\pi} \min\{x,\pi-x\}$ for $x \in [0,\pi]$, we apply the Lagrange mean value theorem to deduce we obtain $\big\|\frac{\hat{u}_{\tau}}{\sin(x)}\big\|_{L^{\infty}} \lesssim \|\hat{u}_{\tau,x}\|_{L^{\infty}}$. For the term $\left\langle \hat{u} \hat{\omega}_{\tau,xx}, \hat{\omega}_{\tau,x} \rho \right\rangle$, using integration by parts and Lemma \ref{le estimates}, we obtain
	\begin{align*}
		\left\langle \hat{u} \hat{\omega}_{\tau,xx}, \hat{\omega}_{\tau,x} \rho \right\rangle & = -\frac{1}{2} \left\langle \hat{u}_{x}\rho, \hat{\omega}_{\tau,x}^{2}\right\rangle + \frac{1}{2} \big\langle\frac{\hat{u} \cos(\frac{x}{2})}{\sin(\frac{x}{2})} \rho, \hat{\omega}_{\tau,x}^{2} \big\rangle \\
		\lesssim & \|\hat{u}_{x}\|_{L^{\infty}} \|\hat{\omega}_{\tau}\|_{\mathcal{H}}^{2} + \bigg\|\frac{\hat{u}}{\sin(\frac{x}{2})}\bigg\|_{L^{\infty}} \|\hat{\omega}_{\tau}\|_{\mathcal{H}}^{2}\\
		\lesssim & E(\tau)\|\hat{\omega}_{\tau}\|_{\mathcal{H}}^{2}.
	\end{align*}
	It follows that
	\begin{align*}
		\frac{1}{2} \frac{\mathrm{d}}{\mathrm{d} \tau} \|\hat{\omega}_{\tau}(\tau)\|_{\mathcal{H}}^{2} & \leq - (\frac{1}{2} - |1-a|) \|\hat{\omega}_{\tau} (\tau)\|_{\mathcal{H}}^{2} + C E(\tau) \|\hat{\omega}_{\tau} (\tau)\|_{\mathcal{H}}^{2}\\
		& \leq -(\frac{1}{2} - C|1-a|) \|\hat{\omega}_{\tau} (\tau)\|_{\mathcal{H}}^{2}.
	\end{align*}
	
	Finally, combining the prior estimates on $E(\tau)$, we deduce that  $\|\hat{\omega}_{\tau}\|_{\mathcal{H}}$ decays exponentially as \(\tau \to +\infty\). Consequently,  $\hat{\omega} (\tau) + \bar{\omega}$ converges strongly in \(\mathcal{H}\) to a limiting profile $\omega_{a} \in \mathcal{H}$. Moreover, since $\|\hat{\omega}\|_{\mathcal{W}}\le E(\tau)\le C\delta_0$ , we  obtain that $\hat{\omega} (\tau) + \bar{\omega}$ converges weakly in \(\mathcal{W}\) to \(\omega_{a}\) as well. Based on the convergence of $\hat{\omega} (\tau) + \bar{\omega}$, we can show that $\hat{c}_{\omega} + \bar{c}_{\omega}$ converges exponentially to some scalar $c_{\omega,a}$ as $\tau \rightarrow +\infty$.  In view of the odd symmetry of \(\hat{\omega}\), which is preserved in time, it follows that the full profile \(\hat{\omega}(\tau) + \bar{\omega}\) remains odd for all \(\tau > 0\), and hence the limit \(\omega_{a}\) inherits the oddness.
	
   Furthermore, by Poincar\'{e}'s inequality and the uniform bound on \(\|\hat{\omega}\|_{\mathcal{H}}\)
	 , we deduce that
	\begin{align*}
		\|\hat{\omega}\|_{L^{\infty}} \lesssim \|\hat{\omega}_{x}\|_{L^{2}} \lesssim \|\hat{\omega}\|_{\mathcal{H}},
	\end{align*}
	which yields uniform convergence in \(L^{\infty}\), i.e.,
	\begin{align*}
		\hat{\omega} (\tau) + \bar{\omega} \to \omega_{a}, \quad \text{uniformly in \(x\).}
	\end{align*}
	 We thus conclude that $(\omega_{a},c_{\omega,a})$ solves the steady-state equation associated with \ref{dre}, and that $\omega_{a}$ corresponds to a self-similar profile of \ref{PJE}.
	 
The proof of Theorem \ref{them profile} is finished.
	
	\end{proof}

	\section{Blow-up for the Inviscid Proudman-Johnson model with H\"{o}lder continuous initial data}\label{section holder}
	In this section, we establish finite time blow-up for $C^{\alpha}$ initial data to the inviscid Proudman-Johnson model (corresponding to $\nu=0$ and $a=1$). Following the approach in Section \ref{section COSSP}, we consider
	\begin{align}\label{dre1}
		\tilde{\omega}_{\tau}(x,\tau) + \tilde{u}\tilde{\omega}_{x}(x,\tau) = c_{\omega}(\tau)\tilde{\omega} + \tilde{u}_{x}\tilde{\omega}, \quad \tilde{u}_{xx} = \tilde{\omega}, \quad  \tau >0, \ x\in\mathbb{T},
	\end{align}
	where $\tilde{\omega}(x,\tau), \tilde{u}(x,\tau), C_{\omega}(\tau)$ and $t(\tau)$ are same defined in Section \ref{section COSSP}.
	
	\begin{proof}[Proof of Theorem \ref{the holder}]
	
	\textbf{Step 1. Perturbation Equations.}\label{sec holder nonlinear}
	
	In this step, we  construct an approximate steady state $(\bar{\omega}_{\alpha},\bar{u}_{\alpha})$ of \ref{dre1} with $\bar{\omega}_{\alpha}\in C^{\alpha}$ and $\bar{\omega}_{\alpha}\approx -Cx^{\alpha}$ near $x=0$, where $\tilde{\omega}, \tilde{u}$ are same defined in \ref{tilde ompsi} of Section \ref{section COSSP}. A natural choice of $\bar{\omega}_{\alpha}$ is
	\begin{align}\label{appro holder}
		\bar{\omega}_{\alpha} = - \text{sgn} (x)|\sin(x)|^{\alpha}, \quad \bar{c}_{\omega,\alpha} = (\alpha-1)\bar{u}_{\alpha,x}(0),
	\end{align}
	where $\bar{u}_{\alpha}$ is related to $\bar{\omega}_{\alpha}$ via Lemma \ref{le psi expression}. We consider odd perturbations $\hat{\omega},\hat{u}$, and note that the odd symmetry is preserved by the dynamics of \eqref{dre1}. Let
	\begin{align*}
		\tilde{\omega} (x,\tau) = \bar{\omega}_{\alpha} + \hat{\omega} (x,\tau),\, \tilde{u} (x,\tau) = \bar{u}_{\alpha} + \hat{u} (x,\tau),\, c_{\omega} (\tau) = \bar{c}_{\omega,\alpha} + \hat{c}_{\omega,\alpha} (\tau), \quad \tau>0, \ x\in\mathbb{T},
	\end{align*}
	with the normalization condition imposed as
	\begin{align}\label{normalization holder}
		\hat{c}_{\omega,\alpha} (\tau) = (\alpha-1)u_{x}(0,\tau).
	\end{align}
	Under the conditions \ref{appro holder} and \ref{normalization holder}, it follows that $\frac{\bar{\omega}}{x^{\alpha}}$ is fixed, i.e.,
	\begin{align*}
		\lim\limits_{x\rightarrow 0} \frac{\hat{\omega} (x,\tau) + \bar{\omega}_{\alpha}(x)}{x^{\alpha}} = \lim\limits_{x\rightarrow 0} \frac{\hat{\omega} (x,0) + \bar{\omega}_{\alpha}(x)}{x^{\alpha}}.
	\end{align*}
	Consequently, if the initial perturbation $\omega_{0}(x)$ vanishes near $x=0$ with higher order than $x^{\alpha}$, then the perturbation $\omega(x,\tau)$ will retain this decay rate near $x=0$, which ensures that we can perform energy estimate on $\omega$ with the singular weight $\rho$ near $x=0$. The perturbations satisfy the following model
	\begin{align}\label{per nuequ}
		\begin{cases}
			\hat{\omega}_{\tau} = \mathcal{L}_{1} \hat{\omega} + \mathcal{R}_{\alpha} + N_{\alpha} + F_{\alpha}, & \tau >0, \ x \in \mathbb{T}, \\
			\hat{u}_{xx} = \hat{\omega}, & \tau >0, \ x \in \mathbb{T},
		\end{cases}
	\end{align}
	where the  leading order linear part is same as in \ref{definition L}. Define $\omega_{res} = \bar{\omega}_{\alpha} - \bar{\omega}, u_{res} = \bar{u}_{\alpha} - \bar{u}$, then the nonlinear term $N_{\alpha}$, residual term $\mathcal{R}_{\alpha}$ and error term $F_{\alpha}$ are defined as follows
	\begin{align*}
		& \mathcal{R}_{\alpha} = - u_{res} \hat{\omega}_{x} + u_{res,x} \hat{\omega} + \omega_{res} \hat{u}_{x} - \omega_{res,x} \hat{u} + \bar{c}_{\omega,\alpha} \hat{\omega} + \hat{c}_{\omega,\alpha} \bar{\omega}_{\alpha},\\
		& N_{\alpha} = \hat{\omega}  \hat{u}_{x} - \hat{u} \hat{\omega}_{x} + \hat{c}_{\omega,\alpha}, \quad  F_{\alpha}=\bar{\omega}_{\alpha}\bar{u}_{\alpha,x}-\bar{u}_{\alpha}\bar{\omega}_{\alpha,x}+\bar{c}_{\omega,\alpha}\bar{\omega}_{\alpha}.
	\end{align*}
	
	\textbf{Step 2. Nonlinear stability analysis.}
	
	Analogous to \textbf{Step 2} in Section \ref{subsection nonlinear}, we obtain
	\begin{align}\label{eq:estiamte-holder1}
		\frac{1}{2} \frac{\mathrm{d}}{\mathrm{d}\tau}\left\langle \hat{\omega}_{x}, \hat{\omega}_{x} \rho \right\rangle \leq - \frac{1}{2} \left\langle \hat{\omega}_{x}, \hat{\omega}_{x} \rho \right\rangle + C \left\langle \hat{\omega}_{x}, \hat{\omega}_{x} \rho \right\rangle^{3/2} + \left\langle(\mathcal{R}_{\alpha})_{x}, \hat{\omega}_{x} \rho \right\rangle + \left\langle(F_{\alpha})_{x}, \hat{\omega}_{x} \rho \right\rangle.
	\end{align}
	Applying Lemma \ref{le res estimates}, we have
	\begin{align*}
		\left\langle(\mathcal{R}_{\alpha})_{x}, \hat{\omega}_{x} \rho \right\rangle = &\left\langle - u_{rex} \hat{\omega}_{xx} + u_{res,xx} \hat{\omega} + \omega_{res} \hat{u}_{xx} - \omega_{res,xx} \hat{u} +\bar{c}_{\omega,\alpha} \hat{\omega}_{x} + \bar{c}_{\omega,\alpha} \bar{\omega}_{\alpha,x}, \hat{\omega}_{x} \rho \right\rangle\\
		\lesssim & \frac{1}{2} \left\langle u_{res,x} \rho, (\hat{\omega}_{x})^{2}\right\rangle - \frac{1}{2} \big\langle \frac{u_{rex} \cos(\frac{x}{2})}{\sin (\frac{x}{2})} \rho, (\hat{\omega}_{x})^{2} \big\rangle  + \left\langle - \omega_{res,xx} \hat{u} + \hat{c}_{\omega,\alpha} \bar{\omega}_{\alpha,x}, \hat{\omega}_{x} \rho \right\rangle\\
		& + \left(\|\omega_{res}\|_{L^{\infty}} + |\alpha-1|\right) \left\langle \hat{\omega}_{x}, \hat{\omega}_{x} \rho \right\rangle\\
		\lesssim & \left\langle -\omega_{res,xx}\hat{u} + \hat{c}_{\omega,\alpha} \bar{\omega}_{\alpha,x}, \hat{\omega}_{x} \rho \right\rangle + |\alpha-1| \left\langle \hat{\omega}_{x}, \hat{\omega}_{x} \rho \right\rangle.
	\end{align*}
	Direct computations show that
	\begin{align*}
		\hat{c}_{\omega,\alpha} \bar{\omega}_{\alpha,x} & = (\alpha - 1) \hat{u}_{x} (0) \bar{\omega}_{\alpha,x} - \omega_{res,xx} \hat{u}\\
		& = \hat{u}_{x} (0) \big((\alpha - 1) \bar{\omega}_{\alpha,x} - \sin(x) \omega_{res,xx}\big) + \omega_{res,xx} \big(\sin(x) \hat{u}_{x} (0) - \hat{u}\big).
	\end{align*}
	Applying the estimate \ref{cancellation} in Lemma \ref{le res estimates}, we obtain
	\begin{align*}
		\|\left(\hat{u}_{x}(0,\tau)(\alpha-1)\bar{\omega}_{\alpha,x}\right)\|_{L^{2}} \lesssim |\alpha-1| \left\langle \hat{\omega}_{x}, \hat{\omega}_{x}\rho\right\rangle^{1/2}.
	\end{align*}
	Using the estimate \ref{omegares estimate} in Lemma \ref{le res estimates}, we have
	\begin{align*}
		\|\left(\omega_{res,xx}\left(\hat{u}_{x}(0,\tau) \sin(x) - \hat{u}\right)\right) \rho^{1/2} \|_{L^{2}} & \lesssim |\alpha-1| \bigg\||\sin(x)|^{\kappa-1} \frac{\hat{u}_{x}(0,\tau)\sin(x) - \hat{u}}{|\sin(x)|} \rho^{1/2} \bigg\|_{L^{2}}\\
		& \lesssim |\alpha-1| \bigg\|\frac{\hat{u}_{x}(0,\tau)\sin(x) - \hat{u}}{|\sin(x)|}\rho^{1/2}\bigg\|_{L^{\infty}}\\
		&\lesssim |\alpha-1| \left\langle \hat{\omega}_{x}, \hat{\omega}_{x}\rho\right\rangle^{1/2}.
	\end{align*}
	Next we prove $\bigg\|\frac{\hat{u}_{x}(0,\tau)\sin(x) - \hat{u}}{|\sin(x)|}\rho^{1/2}\bigg\|_{L^{\infty}} \lesssim \left\langle\hat{\omega}_{x},\hat{\omega}_{x} \rho\right\rangle^{1/2}$. When $|x|\leq \frac{\pi}{2}$, applying Lemma \ref{le psi expression}, we have
	\begin{align*}
		\big|\frac{\hat{u}_{x}(0,\tau)\sin(x) - \hat{u}}{|\sin(x)|} \rho^{1/2} \big|&\lesssim \big|\frac{\hat{u}_{x}(0,\tau) \sin(x) - \hat{u}}{x^{2}} \big| = \big|\frac{\hat{u}_{x}(0,\tau)(\sin(x) -x) + \hat{u}_{x}(0,\tau) x - \hat{u}}{x^{2}}\big|\\
		& \leq \big|\frac{\hat{u}_{x}(0,\tau)(\sin(x) - x)}{x^{2}}\big| + \big|\frac{\int_{0}^{x}(x-y) \hat{\omega} (y) \mathrm{d}y}{x^{2}} \big| \lesssim \left\langle \hat{\omega}_{x},\hat{\omega}_{x} \rho \right\rangle^{1/2}.
	\end{align*}
	When $|x|\geq \frac{\pi}{2}$, we have
	\begin{align*}
		\big|\frac{\hat{u}_{x}(0,\tau)\sin(x) - \hat{u}}{|\sin(x)|} \rho^{1/2} \big| \leq \big| \frac{\hat{u}_{x}(0,\tau)\sin(x) - \hat{u}}{|\sin(x)|}\big| \leq |\hat{u}_{x}(0,\tau)| + \|\frac{\hat{u}}{\sin(x)}\|_{L^{\infty}} \lesssim  \left\langle \hat{\omega}_{x},\hat{\omega}_{x} \rho \right\rangle^{1/2}.
	\end{align*}
	Therefore, we obtain
	\begin{align}\label{eq:estimate-Ra}
		\left\langle(\mathcal{R}_{\alpha})_{x}, \hat{\omega}_{x} \rho \right\rangle \lesssim |\alpha-1| \left\langle \hat{\omega}_{x}, \hat{\omega}_{x} \rho \right\rangle.
	\end{align}
	
	For the error term, using the odd symmetry of the solution, we focus on the pointwise estimates for $x\geq 0$. Direct computations reveal that
	\begin{align*}
		(F_{\alpha})_{x}&=(\bar{\omega}_{\alpha})^{2}+\bar{c}_{\omega,\alpha}\bar{\omega}_{\alpha,x}-\bar{u}_{\alpha}\bar{\omega}_{\alpha,xx}\\
		& = (\sin(x))^{2\alpha} - \alpha(\alpha-1)\bar{u}_{\alpha,x}(0)(\sin(x))^{\alpha-1}\cos(x)\\ & \quad -\bar{u}_{\alpha} \left(-\alpha(\alpha-1) (\sin(x))^{\alpha-2} \cos^{2}(x) + \alpha(\sin(x))^{\alpha}\right)\\
		& = \alpha(\alpha-1)(\sin(x))^{\alpha-1} \cos(x) \left(-\bar{u}_{\alpha,x}(0) + \bar{u}_{\alpha}\frac{\cos(x)}{\sin(x)}\right) + (\sin(x))^{\alpha} \left(-\alpha\bar{u}_{\alpha} + (\sin(x))^{\alpha}\right).
	\end{align*}
	 Next we will show that $\|(F_{\alpha})_{x}\rho^{1/2}\|_{L	^{\infty}}\lesssim|\alpha-1|$. Since
	 \begin{align*}
	 	\bar{u}_{\alpha,x}(0)-\bar{u}_{\alpha} \frac{\cos(x)}{\sin(x)} & = u_{res,x}(0) + \bar{u}_{x}(0) -u_{res} \frac{\cos(x)}{\sin(x)} - \bar{u}\frac{\cos(x)}{\sin(x)}\\
	 	& = (1-\cos(x)) + u_{res,x}(0) - u_{res} \frac{\cos(x)}{\sin(x)}\\
	 	& = (1-\cos(x)) (1+\frac{u_{res}}{\sin(x)}) - \frac{u_{res}}{\sin(x)}+u_{res,x}(0).
	 \end{align*}
	  It follows that
	 \begin{align*}
	 	\bigg\|\left(\bar{u}_{\alpha,x}(0) - \bar{u}_{\alpha}\frac{\cos(x)}{\sin(x)}\right)\rho^{1/2}\bigg\|_{L^{\infty}} \lesssim \bigg\|\frac{1-\cos(x)}{x}(1+\frac{u_{res}}{\sin(x)})\bigg\|_{L^{\infty}} \lesssim 1
	 \end{align*}
	 and
	 \begin{align*}
	 	\bigg\|\left(-\frac{u_{res}}{\sin(x)} + u_{res,x}(0)\right) \rho^{1/2}\bigg\|_{L^{\infty}} = \bigg\|\frac{u_{res,x}(0)\sin(x) - u_{res}}{\sin(x)} \rho^{1/2}\bigg\|_{L^{\infty}} \lesssim 1.
	 \end{align*}
	 On the other hand, using Lemma \ref{le res estimates}, we obtain
	 \begin{align*}
	 	\|\alpha\bar{u}_{\alpha}-(\sin(x))^{\alpha}\|_{L^{\infty}} & = \|(\alpha-1)\bar{u}_{\alpha} + \bar{u}_{\alpha} - \sin(x) + \sin(x) - (\sin(x))^{\alpha}\|_{L^{\infty}}\\
	 	& \leq |\alpha-1| \|\bar{u}_{\alpha}\|_{L^{\infty}} + \|u_{res}\|_{L^{\infty}} + \|\sin(x)-(\sin(x))^{\alpha}\|_{L^{\infty}}\\
	 	& \lesssim |\alpha-1|.
	 \end{align*}
	 Consequently, we get
	 \begin{align}\label{eq:estimate-Fa}
	 	\|(F_{\alpha})_{x}\rho^{1/2}\|_{L^{\infty}}\lesssim & |\alpha-1| \|(\sin(x))^{\alpha-1}\|_{L^{2}} \bigg\|\left(\bar{u}_{\alpha,x}(0)-\bar{u}_{\alpha}\frac{\cos(x)}{\sin(x)}\right) \rho^{1/2} \bigg\|_{L^{\infty}}\nonumber\\
	 	& + \|(\sin(x))^{\alpha} \rho^{1/2}\|_{L^{2}} \|\alpha\bar{u}_{\alpha} + (\sin(x))^{\alpha}\|_{L^{\infty}}\\
	 	\lesssim & |\alpha-1|.\nonumber
	 \end{align}
	
	 \textbf{Step 3. Finite time blow-up analysis.}
	
	 Collecting \ref{eq:estiamte-holder1}, \ref{eq:estimate-Ra} and \ref{eq:estimate-Fa}, we obtain
	 \begin{align}\label{eq:estimate-holder wx}
	 \frac{1}{2} \frac{\mathrm{d}}{\mathrm{d} \tau} \left\langle \hat{\omega}_{x}, \hat{\omega}_{x} \rho \right\rangle \leq - (\frac{1}{2} - C|\alpha-1|) \left\langle \hat{\omega}_{x},\hat{\omega}_{x}\rho\right\rangle + C\left\langle \hat{\omega}_{x}, \hat{\omega}_{x} \rho\right\rangle^{3/2} + |\alpha-1| \left\langle\hat{\omega}_{x},\hat{\omega}_{x}\rho\right\rangle^{1/2}.
	 \end{align}
	 Analogous to \textbf{Step 3} in Section \ref{subsection blowup}, we perform the bootstrap argument to conclude the solution to \ref{PJE} blows-up in finite time for $\nu=0$ with H\"{o}lder initial data, as stated in Theorem \ref{the holder}.
	
	\textbf{Step 4. Convergence to the 	self-similar profile.}
	
	Following \textbf{Step 4} in Section \ref{step3:convergence}, we first establish a weighted higher-order estimate for $\hat{\omega}$ before get the exponentially decay of $\|\hat{\omega}_{\tau}\|_{\mathcal{H}}$ as $\tau \to +\infty$.
	
	Applying the estimates \ref{eq:estimate-L1} and \ref{eq:estimate-Nw} yield
	\begin{align}\label{eq:estimate-holder Dxwx}
		\begin{aligned}
			\frac{1}{2} \frac{\mathrm{d}}{\mathrm{d} \tau} \|\partial_{x}D_{x} \hat{\omega}\rho^{1/2}\|_{L^{2}}^{2} \leq & -\frac{1}{4} \|\partial_{x}D_{x} \hat{\omega} \rho\|_{L^{2}}^{2} + \|\hat{\omega}\|_{\mathcal{H}}^{2}  + C\|\hat{\omega}\|_{\mathcal{H}}^{2} + C\|\partial_{x}D_{x} \hat{\omega}\rho^{1/2}\|_{L^{2}}^{2} \|\hat{\omega}\|_{\mathcal{H}}\\
			&  + \left\langle\partial_{x}D_{x}\mathcal{R}_{\alpha},\partial_{x}D_{x} \hat{\omega}\rho\right\rangle + \left\langle\partial_{x}D_{x}F_{\alpha}, \partial_{x}D_{x}\hat{\omega}\rho\right\rangle.
		\end{aligned}
	\end{align}
	Direct computations show that
	\begin{align}\label{eq:DxRa}
		\begin{aligned}
			\partial_{x}D_{x}\mathcal{R}_{\alpha} = & - \partial_{x}(\sin(x) u_{res} \hat{\omega}_{xx}) + \cos(x) \omega_{res} \hat{\omega} + \sin(x) \omega_{res,x} \hat{\omega} + \sin(x) \omega_{res} \hat{\omega}_{x}\\
			& + \cos(x) \omega_{res} \hat{\omega} +\sin(x) \omega_{res,x} \hat{\omega} +\sin(x) \omega_{res} \hat{\omega}_{x}\\
			& - \sin(x) \omega_{res,xxx} \hat{u} - \sin(x) \omega_{res,xx} \hat{u}_{x} + \hat{c}_{\omega,\alpha} \sin(x) \bar{\omega}_{\alpha,xx}\\
			& - \cos(x)\omega_{res,xx} \hat{u} + \hat{c}_{\omega,\alpha} \cos(x) \bar{\omega}_{\alpha,x} + \bar{c}_{\omega,\alpha} \partial_{x}D_{x} \hat{\omega}.
		\end{aligned}
	\end{align}
	Since
	\begin{align*}
		\sin(x) u_{res} \hat{\omega}_{xx} = u_{res}D_{x} \hat{\omega}_{x} = u_{res}\partial_{x}D_{x} \hat{\omega} - \cos(x) u_{res} \hat{\omega}_{x},
	\end{align*}
	applying Lemma \ref{le res estimates}, we deduce
	\begin{align*}
		&\left\langle-\partial_{x}D_{x} (\sin(x) u_{res} \omega_{xx}),\partial_{x}D_{x} \hat{\omega} \rho\right\rangle\\
		= & \left\langle - u_{res} \partial_{x}\partial_{x}D_{x}\hat{\omega} - u_{res,x} \partial_{x}D_{x} \hat{\omega} - u_{res} \sin(x) \hat{\omega}_{x} + u_{res,x}\cos(x) \hat{\omega}_{x} - u_{res} \cos(x) \hat{\omega}_{xx}, \partial_{x}D_{x} \hat{\omega} \rho\right\rangle\\
		\lesssim & \|u_{res,x}\|_{L^{\infty}} \|\hat{\omega}\|_{\mathcal{W}}^{2} \lesssim |\alpha-1| \|\hat{\omega}\|_{\mathcal{W}}^{2}.
	\end{align*}
	Note that
	\begin{align}\label{eq:DxRa1}
		\begin{aligned}
			&-\sin(x) \omega_{res,xxx} \hat{u} + \hat{c}_{\omega,\alpha} \sin(x) \bar{\omega}_{\alpha,xx} -\sin(x) \omega_{res,xx} \hat{u}_{x}\\
			= & \sin(x)  \hat{u}_{x}(0,\tau) \partial_{x}\left((\alpha-1)\bar{\omega}_{\alpha,x} - \sin(x)\omega_{res,xx}\right)- \sin(x)(\omega_{res,xxx} \hat{u} + \omega_{res,xx} \hat{u}_{x}) \\
			& + \sin(x) (\hat{u}_{x}(0,\tau)\cos(x) \omega_{res,xx} + \hat{u}_{x}(0,\tau) \sin(x) \omega_{res,xxx})\\
			= & \sin(x)  \hat{u}_{x}(0,\tau) \partial_{x} \left((\alpha-1)\bar{\omega}_{\alpha,x} - \sin(x) \omega_{res,xx}\right) + \sin(x) \partial_{x} \left(\hat{u}_{x}(0,\tau) \sin(x) \omega_{res,xx}\right)\\
			& - \sin(x)\partial_{x}(\omega_{res,xx} \hat{u})\\
			= & \sin(x) \hat{u}_{x}(0,\tau) \partial_{x} \left((\alpha-1)\bar{\omega}_{\alpha,x} - \sin(x) \omega_{res,xx}\right) + \sin(x) \partial_{x} \left((\hat{u}_{x}(0,\tau)\sin(x) - \hat{u})\omega_{res,xx}\right)
		\end{aligned}
	\end{align}
	and
	\begin{align}\label{eq:DxRa2}
		& \quad -\cos(x) \omega_{res,xx} \hat{u} + \hat{c}_{\omega,\alpha} \cos(x) \bar{\omega}_{\alpha,x}\nonumber\\
		& = \cos(x) \left(-\omega_{res,xx} \hat{u} + (\alpha-1) \hat{u}_{x}(0,\tau)\bar{\omega}_{\alpha,x}\right)\\
		& = \cos(x) \left(\hat{u}_{x}(0,\tau) \left((\alpha-1) \bar{\omega}_{\alpha,x} -\sin(x)\omega_{res,xx}\right) + \omega_{res,xx}\left(\hat{u}_{x}(0,\tau)\sin(x) - \hat{u} \right)\right).\nonumber
	\end{align}
	Applying Lemma \ref{le res estimates} and Lemma \ref{le psi expression}, it follows that
	\begin{align}\label{eq:DxRa3}
		\begin{aligned}
			&\|\left(-\sin(x) \omega_{res,xxx}\hat{u} - \sin(x) \omega_{res,xx} \hat{u}_{x} + \hat{c}_{\omega,\alpha} \sin(x) \bar{\omega}_{\alpha,xx}\right) \rho^{1/2} \|_{L^{2}}\\
			\lesssim & |\alpha-1| \|\hat{\omega}\|_{\mathcal{H}} + \|\left(\sin(x) \partial_{x} \left((\hat{u}_{x}(0,\tau)\sin(x) - \hat{u}) \omega_{res,xx}\right)\right)\rho^{1/2}\|_{L^{2}}\\
			\lesssim & |\alpha-1| \|\hat{\omega}\|_{\mathcal{H}} + |\alpha-1|\bigg\||\sin(x)|^{\kappa-1} \frac{\hat{u}_{x}(0,\tau)\sin(x) - \hat{u}}{|\sin(x)|} \rho^{1/2} \bigg\|_{L^{2}}\\
			& + |\alpha-1| \||\sin(x)|^{\kappa-1} \left(\hat{u}_{x}(0,\tau) - \hat{u}_{x}\right) \rho^{1/2}\|_{L^{2}}\\
			\lesssim & |\alpha-1| \|\hat{\omega}\|_{\mathcal{H}} + |\alpha-1| \bigg\|\frac{\hat{u}_{x}(0,\tau) \sin(x) - \hat{u}}{|\sin(x)|} \rho^{1/2} \bigg\|_{L^{\infty}} + |\alpha-1| \bigg\|\frac{\hat{u}_{x}(0,\tau) - \hat{u}_{x}}{x}\bigg\|_{L^{\infty}}\\
			\lesssim & |\alpha-1|\|\hat{\omega}\|_{\mathcal{H}}.
		\end{aligned}
	\end{align}
	Combining with \ref{eq:DxRa}-\ref{eq:DxRa3} yields
	\begin{align}\label{eq:estimate-DxRa}
		\left\langle\partial_{x}D_{x}\mathcal{R}_{\alpha},\partial_{x}D_{x} \hat{\omega} \rho \right\rangle \lesssim |\alpha-1| (\|\hat{\omega}\|_{\mathcal{H}} + \|\hat{\omega}\|_{\mathcal{W}}) \|\hat{\omega}\|_{\mathcal{W}} \lesssim  |\alpha-1| \|\hat{\omega}\|_{\mathcal{W}}^{2}.
	\end{align}
	Moreover, for $x\geq 0$, we have
	\begin{align*}
		\bar{\omega}_{\alpha,xxx}  = & -\alpha(\alpha-1)(\alpha-2) (\sin(x))^{\alpha-3}\cos^{3}(x) + 2\alpha(\alpha-1)(\sin(x))^{\alpha-1} \cos(x) \\
		& + \alpha^{2}(\sin(x))^{\alpha-1}.
	\end{align*}
	Similar to the estimate \ref{eq:estimate-Fa}, we get
	\begin{align*}
		\|\partial_{x}D_{x}F_{\alpha}\rho^{1/2}\|_{L^{2}}\leq & \|\left((\bar{\omega}_{\alpha})^{2} + \bar{c}_{\omega,\alpha}\bar{\omega}_{\alpha,x} - \bar{u}_{\alpha}\bar{\omega}_{\alpha,xx}\right)\rho^{1/2}\|_{L^{2}}\\
		& + \|\sin(x) \left(2\bar{\omega}_{\alpha}\bar{\omega}_{\alpha,x} + \bar{c}_{\omega,\alpha}\bar{\omega}_{\alpha,xx}-\bar{u}_{\alpha,x}\bar{\omega}_{\alpha,xx}-\bar{u}_{\alpha}\bar{\omega}_{\alpha,xxx}\right) \rho^{1/2} \|_{L^{2}}\\
		\lesssim & |\alpha-1|,
	\end{align*}
	which implies
	\begin{align}\label{eq:estimate-DxFa}
		\left\langle\partial_{x}D_{x}F_{\alpha},\partial_{x}D_{x} \hat{\omega} \rho \right\rangle \lesssim |\alpha-1| \left\langle\partial_{x}D_{x} \hat{\omega}, \partial_{x}D_{x} \hat{\omega} \rho \right\rangle^{1/2} \lesssim  |\alpha-1| \|\hat{\omega}\|_{\mathcal{W}}.
	\end{align}
	Combining \ref{eq:estimate-holder wx}, \ref{eq:estimate-holder Dxwx}, \ref{eq:estimate-DxRa} and \ref{eq:estimate-DxFa}, we deduce
	\begin{align*}
		\frac{\mathrm{d}}{\mathrm{d} \tau} E(\tau) \leq - (\frac{1}{8} - C|\alpha-1|) E(\tau) + CE^{2}(\tau) + C|\alpha-1|,
	\end{align*}
	where $E(\tau)^{2} = \|\partial_{x}D_{x} \hat{\omega}\rho^{1/2}\|_{L^{2}}^{2} + \mu \|\hat{\omega}\|_{\mathcal{H}}^{2}$. Using the bootstrap argument on the energy $E(\tau)$, we obtain there exist a sufficiently small constant $\delta_{2}> 0$ and a constant $C>0$ such that, if $|1-a| < \delta_{2}$ and $E(0) < C\delta_{2}$, then $E(\tau) \leq C\delta_{2}$ for all $\tau \geq 0$. Following the same strategy as \textbf{Step 4} in Section \ref{step3:convergence}, we can get the analogous estimate of $\left\langle\omega_{\tau},\omega_{\tau}\right\rangle_{\mathcal{H}}$ and establish exponential convergence of $\left\langle\omega_{\tau},\omega_{\tau}\right\rangle_{\mathcal{H}}$ to zero. The proof of Theorem \ref{the holder} is then completed.
     \end{proof}
	
	\section{Blow-up for the Viscous Model}\label{section viscous}
	In this section, we establish the finite time blow-up for the viscous model, which can be written as
	\begin{align}\label{nuPJE}
		\begin{cases}
			\omega_{t} + u\omega_{x} = a\omega u_{x} + \nu\omega_{xx}, & t >0, \ x \in \mathbb{T}, \\
			u_{xx} =\omega, & t >0, \ x \in \mathbb{T},
		\end{cases}
	\end{align}
	 with $a > 1$ and fixed $\nu>0$.
	
	We introduce the rescaled variables
	\begin{align*}
		\tilde{\omega}(x,\tau)=C_{\omega}(\tau)\omega(x,t(\tau)), \quad \tilde{u}=C_{\omega}(\tau)u(x,t(\tau)),
	\end{align*}
	where
	\begin{align*}
		C_{\omega}(\tau)=C_{\omega}(0)\exp\left(\int_{0}^{\tau}c_{\omega}(s)ds\right), \quad t(\tau)=\int_{0}^{\tau}C_{\omega}(s)ds,
	\end{align*}
	where $C_{\omega}(0)>0$. Compared with the inviscid case, the positive constant $C_{\omega}(0)$ provides an additional degree of freedom. This flexibility allows us to choose $C_{\omega}(0)$ in such a way that the viscous term $\nu C_{\omega}(\tau) \tilde{\omega}_{xx}$ remains small relative to the dominant terms in the rescaled system \eqref{tildenuPJE}. We will choose $C_{\omega}(0)=|1-a|^2$ in our proof.

	In terms of these variables, \ref{nuPJE} becomes
	\begin{align}\label{tildenuPJE}
		\begin{cases}
			\tilde{\omega}_{\tau} + \tilde{u}\tilde{\omega}_{x} = a\tilde{u}_{x}\tilde{\omega} + c_{\omega}\tilde{\omega} + \nu C_{\omega}(\tau)\tilde{\omega}_{xx}, & \tau >0, \ x \in \mathbb{T}, \\
			\tilde{u}_{xx} =\tilde{\omega}, & \tau >0, \ x \in \mathbb{T}.
		\end{cases}
	\end{align}
	To prove Theorem \ref{the viscous}, it suffices to establish the dynamical stability of \ref{tildenuPJE} under the scaling parameter $c_{\omega}<-c<0$ for all time, as discussed in \cite{HW2024,CHH2021}.

	\begin{proof}[Proof of Theorem \ref{the viscous}]
		\textbf{Step 1. Perturbation equations.}
		
		We now consider the approximate steady state
		\begin{align*}
			\bar{\omega}=-\sin x, \quad \bar{u}=\sin x, \quad \bar{c}_{\omega}=(1-a)\bar{u}_{x}(0)-\nu C_{\omega}(\tau)\frac{\bar{\omega}_{xxx}(0)}{\bar{\omega}_{x}(0)}=(1-a)+\nu C_{\omega}(\tau).
		\end{align*}
		Let
		\begin{align*}
			\tilde{\omega} (x,\tau) = \bar{\omega} + \hat{\omega} (x,\tau), \, \tilde{u} (x,\tau) = \bar{u} + \hat{u} (x,\tau), \, c_{\omega} (\tau) = \bar{c}_{\omega} + \hat{c}_{\omega} (\tau), \quad   \tau >0, \ x\in\mathbb{T},
		\end{align*}
		where $\hat{\omega}, \hat{u}$ and $\hat{c}_{\omega}$ are odd perturbations. The odd symmetry of the solution is preserved in time by  \ref{tildenuPJE}. We also introduce the following normalization condition
		\begin{align*}
			\hat{c}_{\omega}(\tau) = (1-a)\hat{u}_{x}(0,\tau)-\nu C_{\omega}(\tau) \frac{\hat{\omega}_{xxx}(0,\tau)}{\bar{\omega}_{x}(0)} = (1-a)\hat{u}_{x}(0,\tau) + \nu C_{\omega}(\tau) \hat{\omega}_{xxx}(0,\tau).
		\end{align*}
		This normalization ensures that $\hat{\omega}_{x}(0,\tau)$ remains zero, provided that the initial perturbation satisfies $\hat{\omega}_{x}(0,0)=0.$  Indeed, if $\hat{\omega}_{x}(0,\tau)=0,$ it then follows that
		\begin{align*}
			\frac{\mathrm{d}}{\mathrm{d}\tau} \hat{\omega}_{x} (0,\tau) = &\frac{\mathrm{d}}{\mathrm{d}\tau}
			\left(\hat{\omega}_{x} (0,\tau) + \bar{\omega}_{x}(0)\right)\\
			= & (a-1) \left(\hat{u}_{x}(0,\tau) + \bar{u}_{x}(0)\right) \left(\hat{\omega}_{x}(0,\tau) + \bar{\omega}_{x}(0)\right) + \nu C_{\omega}(\tau)\left(\hat{\omega}_{xxx} (0,\tau) + \bar{\omega}_{xxx}(0) \right)\\
			& + \left(\hat{c}_{\omega}(\tau) + \bar{c}_{\omega}\right) \left(\hat{\omega}_{x} (0,\tau) + \bar{\omega}_{x}(0)\right) =0.
		\end{align*}
		The perturbation equations can then be written as
		\begin{align}\label{per nuequ}
			\begin{cases}
				\hat{\omega}_{\tau} = \mathcal{L}_{1} \hat{\omega} + (1-a) \mathcal{A}\hat{\omega} + N(\hat{\omega}) + F(\bar{\omega}) + \nu C_{\omega}(t) V, & \tau >0, \ x \in \mathbb{T}, \\
				\hat{u}_{xx} = \hat{\omega}, & \tau >0, \ x \in \mathbb{T},
			\end{cases}
		\end{align}
		where the terms $\mathcal{L}_{1} \hat{\omega},\mathcal{A} \hat{\omega}, N(\hat{\omega}), F(\bar{\omega})$ are defined exactly as in the inviscid case, and the term
		\begin{align*}
			V = \hat{\omega}_{xx} + \left( 1 + \hat{\omega}_{xxx}(0,\tau)\right)\hat{\omega} - \hat{\omega}_{xxx}(0,\tau)\sin(x),
		\end{align*}
		 corresponds to  the effect of viscosity.

		\textbf{Step 2. Stability analysis of the rescaled perturbation equation.}

		Analogous to the estimates in the inviscid case, we obtain
		\begin{align}
			\frac{1}{2}\frac{\mathrm{d}}{\mathrm{d}\tau} \|\hat{\omega}\|_{\mathcal{H}}^{2} \leq - (\frac{1}{2} - C|1-a|) \|\hat{\omega}\|_{\mathcal{H}}^{2} + C\|\hat{\omega}\|_{\mathcal{H}}^{3} + C|1-a| \|\hat{\omega}\|_{\mathcal{H}} + \nu C_{\omega}(\tau) \left\langle V_{x}, \hat{\omega}_{x}\rho\right\rangle.
		\end{align}
		To handle the viscous term, we write
		\begin{align*}
			\left\langle V_{x},\hat{\omega}_{x} \rho \right\rangle & = \left\langle \hat{\omega}_{xxx} + \left( 1 + \hat{\omega}_{xxx}(0,\tau) \right)\hat{\omega}_{x} - \hat{\omega}_{xxx}(0,\tau) \cos(x) , \hat{\omega}_{x} \rho \right\rangle\\
			& = \left\langle\hat{\omega}_{xxx} - \hat{\omega}_{xxx}(0,\tau), \hat{\omega}_{x} \rho \right\rangle + \hat{\omega}_{xxx}(0,\tau) \left\langle 1-\cos(x),\hat{\omega}_{x}\rho\right\rangle + \left(1+\hat{\omega}_{xxx}(0,\tau)\right)\left\langle\hat{\omega}_{x},\hat{\omega}_{x}\rho\right\rangle.
		\end{align*}
		Notice that, when \(x\neq 0\),
		\begin{align*}
			|\rho_{x}| = \bigg|\frac{1}{2\pi} \frac{-\sin(x)}{(1-\cos(x))^{2}} \bigg| = \bigg|\rho\frac{-\sin(x)}{1-\cos(x)}\bigg|\lesssim\rho|\frac{1}{x}|
		\end{align*}
		and
		\begin{align*}
			|\rho_{xx}|=\bigg|\frac{1}{2\pi}\frac{2\sin^{2}(x) - \cos(x) (1-\cos(x))}{(1-\cos(x))^{3}}\bigg| = \rho\bigg|\frac{2\sin^{2}(x)}{(1-\cos(x))^{2}} - \frac{\cos(x)}{1-\cos(x)} \bigg| \lesssim \rho \frac{1}{x^{2}}.
		\end{align*}
		Use integration by parts twice to yield
		\begin{align*}
			\left\langle \hat{\omega}_{xxx} - \hat{\omega}_{xxx}(0,\tau),\hat{\omega}_{x}\rho\right\rangle = & - \left\langle\hat{\omega}_{xx} - \hat{\omega}_{xxx}(0,\tau) x, \hat{\omega}_{xx} \rho \right\rangle - \left\langle\hat{\omega}_{xx} - \hat{\omega}_{xxx}(0,\tau) x, \hat{\omega}_{x} \rho_{x}\right\rangle\\
			= & - \left\langle \hat{\omega}_{xx} - \hat{\omega}_{xxx}(0,\tau) x, \hat{\omega}_{xx} \rho \right\rangle -\big\langle \hat{\omega}_{xx} - \hat{\omega}_{xxx}(0,\tau) x, \frac{x^{2}}{2} \hat{\omega}_{xxx}(0,\tau)\rho_{x}\big\rangle\\
			& - \big\langle\hat{\omega}_{xx} - \hat{\omega}_{xxx}(0,\tau) x, (\hat{\omega}_{x}-\frac{x^{2}}{2} \hat{\omega}_{xxx}(0,\tau)) \rho_{x} \big\rangle\\
			\leq & - \|\hat{\omega}_{xx} \rho^{1/2} \|_{L^{2}}^{2} + C \left(|\hat{\omega}_{xxx}(0,\tau)|  \|\hat{\omega}_{xx}\rho^{1/2}\|_{L^{2}} + |\hat{\omega}_{xxx}(0,\tau)|^{2}\right)\\
			& - \frac{1}{2} \big\langle((\hat{\omega}_{x} - \frac{x^{2}}{2})^{2})_{x},\rho_{x}\big\rangle\\
			\leq & - \|\hat{\omega}_{xx}\rho^{1/2}\|_{L^{2}}^{2}\\
			& + C\left(|\hat{\omega}_{xxx}(0,\tau)|  \|\hat{\omega}_{xx}\rho^{1/2}\|_{L^{2}} + |\hat{\omega}_{xxx}(0,\tau)|^{2} + \|(\frac{\hat{\omega}_{x}}{x} - \frac{x}{2} \hat{\omega}_{xxx}(0,\tau)) \rho^{1/2}\|_{L^{2}}^{2}\right)\\
			\leq & - \frac{1}{2}\|\hat{\omega}_{xx}\rho^{1/2}\|_{L^{2}}^{2} + C\left(|\hat{\omega}_{xxx}(0,\tau)|^{2} + \|\frac{\hat{\omega}_{x}}{x} \rho^{1/2} \|_{L^{2}}^{2}\right).
		\end{align*}
		Thus, we obtain
		\begin{align}\label{V omegax}
			\left\langle V_{x},\hat{\omega}_{x}\rho\right\rangle \leq - \frac{1}{2} \|\hat{\omega}_{xx} \rho^{1/2}\|_{L^{2}}^{2} + C\left(|\hat{\omega}_{xxx}(0,\tau)|^{2} + \|\frac{\hat{\omega}_{x}}{x}\rho^{1/2}\|_{L^{2}}^{2} + (1 + |\hat{\omega}_{xxx}(0,\tau)|  \|\hat{\omega}\|_{\mathcal{H}}^{2}\right).
		\end{align}
		The essential difficulty of the viscous terms is that after integration by parts, the singular weight produces various trouble terms. Fortunately, these terms contribute only to higher-order terms near the origin. Denoting the interval $I=[-\frac{\pi}{2},\frac{\pi}{2}]$, and using $\hat{\omega}_{x} (0) = \hat{\omega}_{xx} (0) = 0$, we obtain
		\begin{align}\label{positive term}
			\begin{aligned}
				\|\frac{\hat{\omega}_{x}}{x}\rho^{1/2}\|_{L^{2}}^{2} = & \int_{I} \frac{\hat{\omega}_{x}^{2}}{x^{2}} \rho \mathrm{d}x + \int_{[-\pi,\pi] \setminus I} \frac{\hat{\omega}_{x}^{2}}{x^{2}}\rho \mathrm{d}x \\
				\lesssim & \|\hat{\omega}_{xxx}\|_{L^{\infty}(I)}^{2}  + \|\hat{\omega}\|_{\mathcal{H}}^{2}.
			\end{aligned}
		\end{align}
		Combing \ref{V omegax} with \ref{positive term}, we obtain
		\begin{align}\label{estimate Vx}
			\left\langle V_{x},\hat{\omega}_{x} \rho \right \rangle\leq - \frac{1}{2} \|\hat{\omega}_{xx} \rho^{1/2} \|_{L^{2}}^{2} + C\left(\|\hat{\omega}_{xxx} \|_{L^{\infty}(I)}^{2} + (1+\|\hat{\omega}_{xxx}\|_{L^{\infty}(I)})\|\hat{\omega}\|_{\mathcal{H}}^{2}\right).
		\end{align}
		
		\textbf{Step 3. Higher-order estimates of $\hat{\omega}$.}
		
		  In the above analysis, we need an appropriate higher-order norm to close the estimates.  Specifically, we aim to extract damping from the leading-order linear term and bound the term $\hat{\omega}_{xxx}$ by interpolating between lower-order and higher-order norms, using the Gagliardo-Nirenberg inequality. To this end, we define the $k$-th order weighted norms for $k\geq 1$ as
		\begin{align*}
			E_{k}^{2}(\tau) := \left\langle \hat{\omega}^{(k+1)},\hat{\omega}^{(k+1)}\rho_{k}\right\rangle, \quad \rho_{k} := (1+\cos(x))^{k},
		\end{align*}
		where $\omega^{(k)} := \partial_{x}^{k}\omega$. For clarity, we define $\rho_{0} = \rho$ and $E_{0}^{2}(\tau) = \left\langle \hat{\omega}_{x},\hat{\omega}_{x}\rho \right\rangle$. Similar weighted higher-order norms have been also used in \cite{HW2024}.
		
		Now, we begin with estimating $E_{k}^{2}(\tau)$ for $k>0$ as follows
		\begin{align*}
			\frac{1}{2} \frac{\mathrm{d}}{\mathrm{d}\tau} E_{k}^{2}(\tau) = & \left\langle(\mathcal{L}_{1}\hat{\omega})^{(k+1)}, \hat{\omega}^{(k+1)} \rho_{k}\right\rangle +  (1-a) \left\langle(\mathcal{A} \hat{\omega})^{(k+1)}, \hat{\omega}^{(k+1)} \rho_{k}\right\rangle +  \left\langle(N(\hat{\omega}))^{(k+1)},\hat{\omega}^{(k+1)} \rho_{k} \right\rangle\\
			& + \left\langle(F(\bar{\omega}))^{(k+1)},\hat{\omega}^{(k+1)} \rho_{k} \right\rangle + \nu C_{\omega}(t) \left\langle V^{(k+1)}, \hat{\omega}^{(k+1)}\rho_{k}\right\rangle.
		\end{align*}
		We proceed by analyzing the linear terms and extracting the damping. For simplification, we denote the terms as lower order terms (l.o.t for short) if their $\rho_{k}$-weighted $L^{2}$-norms are bounded by $\sum_{i=0}^{k-1}E_{i}(\tau)$. Since $\rho_{k}\leq C(k)\rho_{i}$ for $i<k$, using the l.o.t notation, we focus solely on the higher-order terms.
		\begin{align*}
			\left\langle(\mathcal{L}_{1}\hat{\omega})^{(k+1)}, \hat{\omega}^{(k+1)} \rho_{k} \right\rangle = \left\langle - \sin(x) \hat{\omega}^{(k+2)} - k \cos(x) \hat{\omega}^{(k+1)} + \mathrm{l.o.t.}, \hat{\omega}^{(k+1)} \rho_{k}\right\rangle.
		\end{align*}
		Using integration by parts yields
		\begin{align*}
			\left\langle - \sin(x) \hat{\omega}^{(k+2)} - k\cos (x)\omega^{(k+1)}, \hat{\omega}^{(k+1)} \rho_{k} \right\rangle = & \big\langle  \cos(x) \rho_{k} - k \rho_{k-1} \sin^{2}(x), \frac{1}{2}(\hat{\omega}^{(k+1)})^{2}\big\rangle\\
			& - k \left\langle\cos(x) \hat{\omega}^{(k+1)},\hat{\omega}^{(k+1)}\rho_{k}\right\rangle\\
			= & \big\langle(\hat{\omega}^{(k+1)})^{2}, \frac{1}{2} \cos(x) \rho_{k} - \frac{1}{2} k \rho_{k-1} \sin^{2}(x) - k\cos(x) \rho_{k} \big\rangle\\
			= & \big\langle(\hat{\omega}^{(k+1)})^{2},\rho_{k} (-\frac{k}{2}-\frac{(k-1)}{2} \cos(x)) \big\rangle\\
			\leq & - \left\langle \hat{\omega}^{(k+1)}, \hat{\omega}^{(k+1)} \rho_{k} \right\rangle.
		\end{align*}
		It follows that
		\begin{align*}
			\left\langle(\mathcal{L}_{1} \hat{\omega})^{(k+1)}, \hat{\omega}^{(k+1)} \rho_{k} \right\rangle & \leq -E_{k}^{2}(\tau) + C(k)\sum_{i=0}^{k-1} E_{i}(\tau) E_{k}(\tau)\\
			& \leq - \frac{1}{2} E_{k}^{2}(\tau) + C(k)\sum_{i=0}^{k-1} E_{i}^{2} (\tau).
		\end{align*}
		Similar estimates give
		\begin{align*}
			\left\langle(\mathcal{A} \hat{\omega})^{(k+1)}, \hat{\omega}^{(k+1)} \rho_{k}\right\rangle \leq E_{k}^{2} (\tau) + C(k) \sum_{i=0}^{k-1} E_{i}^{2} (\tau)
		\end{align*}
		and
		\begin{align*}
			\left\langle(F(\bar{\omega}))^{(k+1)},\hat{\omega}^{(k+1)} \rho_{k} \right\rangle\leq |1-a| C(k) E_{k} (\tau).
		\end{align*}
		Since
		\begin{align*}
			(N(\hat{\omega}))^{(k+1)} = -(\hat{u} \hat{\omega}_{x})^{(k+1)} + a(\hat{u}_{x} \hat{\omega})^{(k+1)} +(1-a) \hat{u}_{x} (0) \hat{\omega}^{(k+1)}.
		\end{align*}
		For the term $\hat{\omega}^{(k+2)}\hat{u}$, we can apply integration by parts and Poincar\'{e}'s inequality to obtain
		\begin{align*}
			\left\langle \hat{\omega}^{(k+2)} \hat{u},\hat{\omega}^{(k+1)} \rho_{k} \right\rangle & = - \frac{1}{2} \left\langle \hat{u}_{x} \rho_{k} -k \hat{u} \rho_{k-1} \sin(x), (\hat{\omega}^{(k+1)})^{2} \right \rangle\\
			& = \left\langle \hat{u}_{x} \rho_{k}, (\hat{\omega}^{(k+1)})^{2}\right\rangle - k \big\langle\frac{\sin^{2}(x)}{1 + \cos(x)}\frac{\hat{u}}{\sin(x)}, (\hat{\omega}^{(k+1)})^{2} \big\rangle\\
			& \leq C(k) \left(\|\hat{u}_{x}\|_{L^{\infty}} + \big\|\frac{\hat{u}}{\sin(x)}\big\|_{L^{\infty}}\right)\left\langle \hat{\omega}^{(k+1)}, \hat{\omega}^{(k+1)}\rho_{k}\right\rangle\\
			&\leq C(k)E_{0}(\tau)E_{k}^{2}(\tau).
		\end{align*}
		The terms associated with $\hat{u}_{x} \hat{\omega}^{(k+1)}$ can be handled in a similar way. Concerning the term $\hat{u}^{(i)} \hat{\omega}^{(k+2-i)} (2\leq i\leq k+1)$, we have
		\begin{align*}
			\|\hat{u}^{(i)} \hat{\omega}^{(k+2-i)} \rho_{k}^{1/2}\|_{L^{2}} & = \|\hat{u}^{(i)} \hat{\omega}^{(k+2-i)} (1 + \cos(x))^{1/2}\|_{L^{2}}\\
			& = \|\hat{u}^{(i)} \hat{\omega}^{(k+2-i)}( 1 + \cos(x))^{\frac{k+1-i}{2}}(1+\cos(x))^{\frac{i-1}{2}}\|_{L^{2}}\\
			& \leq \|\hat{u}^{(i)}(1+\cos(x))^{\frac{i-1}{2}}\|_{L^{\infty}} E_{k+1-i}(\tau).
		\end{align*}
		Since $\hat{u}_{xx} = \hat{\omega}$, it concludes that
		\begin{align*}
			|\hat{u}^{(i)}(1+\cos(x))^{\frac{i-1}{2}}|&=\big|\int_{-\pi}^{x}\partial_{x}\left(\hat{u}^{(i)}(1+\cos(x))^{\frac{i-1}{2}}\right) \mathrm{d}x \big|\\
			& \leq \|\hat{u}^{(i+1)}(1+\cos(x))^{\frac{i-1}{2}}\|_{L^{1}} + \frac{i-1}{2}\|\hat{u}^{(i)}(1+\cos(x))^{\frac{i-3}{2}}\sin(x)\|_{L^{1}}\\
			&\leq C(k)\sum_{i=0}^{k}E_{i}(\tau).
		\end{align*}
		Consequently, the nonlinear terms can be estimated as
		\begin{align*}
			\left\langle(N(\hat{\omega}))^{(k+1)},\hat{\omega}^{(k+1)}\rho_{k}\right\rangle\leq C(k)\sum_{i=0}^{k}E_{i}^{2}(\tau)E_{k}(\tau).
		\end{align*}
		
		Next, we focus on the viscous terms. Direct computations reveal that
		\begin{align*}
			\left\langle V^{(k+1)}, \hat{\omega}^{(k+1)} \rho_{k}\right\rangle \leq \left\langle \hat{\omega}^{(k+3)}, \hat{\omega}^{(k+1)}\rho_{k}\right\rangle + C(k)\|\hat{\omega}_{xxx}\|_{L^{\infty}(I)} E_{k}(\tau) + (1 + \|\hat{\omega}_{xxx}\|_{L^{\infty}(I)})E_{k}^{2}(\tau).
		\end{align*}
		For the term $\left\langle \hat{\omega}^{(k+3)},\hat{\omega}^{(k+1)}\rho_{k}\right\rangle$, using integration by parts, we obtain
		\begin{align*}
			\left\langle \hat{\omega}^{(k+3)}, \hat{\omega}^{(k+1)}\rho_{k}\right\rangle & = -\left\langle \hat{\omega}^{(k+2)}, \hat{\omega}^{(k+2)}\rho_{k}\right\rangle +  k \left\langle \hat{\omega}^{(k+2)},\hat{\omega}^{(k+1)} \rho_{k-1}\sin(x) \right\rangle\\
			& = -\left\langle \hat{\omega}^{(k+2)}, \hat{\omega}^{(k+2)} \rho_{k} \right\rangle + \frac{k}{2} \left\langle \hat{\omega}^{(k+1)}, \hat{\omega}^{(k+1)} \rho_{k-1}((k-1) - k\cos(x))\right\rangle\\
			& \leq -\left\langle \hat{\omega}^{(k+2)}, \hat{\omega}^{(k+2)}\rho_{k}\right\rangle + C(k) \left\langle \hat{\omega}^{(k+1)}, \hat{\omega}^{(k+1)}\rho_{k-1}\right\rangle.
		\end{align*}
		Combining with the leading order estimate \ref{estimate Vx}, we conclude that, for small enough constants $0<\mu(k_{0})<1$,
		\begin{align}\label{viscous estimate}
			\sum_{k=0}^{k_{0}}\mu^{k}\left\langle V^{(k+1)},\hat{\omega}^{(k+1)}\rho_{k}\right\rangle \leq C(k_0)\big(\|\hat{\omega}_{xxx}\|_{L^{\infty}(I)}^{2} + (1+\|\hat{\omega}_{xxx}\|_{L^{\infty}(I)})\sum_{k=0}^{k_{0}} \mu^{k} E_{k}^{2} (\tau)\big).
		\end{align}
		Here $\mu(k_{0})$ is a generic constant depending on $k_{0}$. We can choose $k_{0}$ large enough later so that $\|\hat{\omega}_{xxx}\|_{L^{\infty}(I)}$ can be bounded using interpolation inequalities.
		
		\textbf{Step 4. Finite time blow-up analysis.}

		For a fixed $k_{0}$, there exists a sufficiently small constant $0<\mu_{1}(k_{0})<\mu(k_{0})$, such that the following estimate holds
		\begin{align*}
			\frac{\mathrm{d}}{\mathrm{d}\tau} I_{k_{0}}^{2}(\tau)\leq & -(\frac{1}{2}-C|1-a|)I_{k_{0}}^{2}(\tau)+CI_{k_{0}}^{3}(\tau)+C|1-a|I_{k_{0}}(\tau)\\
			& + C\nu C_{\omega}(\tau)\left(\|\hat{\omega}_{xxx}\|_{L^{\infty}(I)}^{2}+(1+\|\hat{\omega}_{xxx}\|_{L^{\infty}(I)})I_{k_{0}}^{2}(\tau)\right),
		\end{align*}
		where the energy is defined as
		\begin{align*}
			I_{k_{0}}^{2}(\tau)=\sum_{k=0}^{k_{0}}\mu_{1}^{k}(k_{0})E_{k}^{2}(\tau).
		\end{align*}
		Here, the constants depend on $k_{0}$ and $\mu,$ but once $k_{0}$ is fixed, with $\mu=\mu_{1}(k_{0})$, they become just constants. We will later make both $C_{\omega}(t)$ and $|1-a|$ sufficiently small to close the argument. By the Gagliardo-Nirenberg inequality, we have
		\begin{align*}
			\|\hat{\omega}_{xxx}\|_{L^{\infty}(I)}\lesssim \|\hat{\omega}^{(4)}\|_{L^{2}(I)}^{\frac{7}{8}} \|\hat{\omega}\|_{L^{2}(I)}^{\frac{1}{8}}.
		\end{align*}
		Consequently, we conclude that
		\begin{align*}
			\|\hat{\omega}_{xxx}\|_{L^{\infty}(I)}\lesssim I_{k_{0}}(\tau),
		\end{align*}
		for $k_{0}\geq 4$. For simplicity, we take $k_{0}=4$ as an example, which implies
		\begin{align*}
			\frac{\mathrm{d}}{\mathrm{d}\tau}I_{4}(\tau) \leq  - (\frac{1}{2} - C|1-a|) I_{4}(\tau) + CI_{4}^{2}(\tau) + C|1-a| + C\nu C_{\omega}(\tau)(1+I_{4}(t))I_{4}(\tau).
		\end{align*}
		We now choose \(C_{\omega} (0) = |1-a|^{2} < \delta_{3}\) for a sufficiently small \(\delta_{3} > 0\) and use a bootstrap argument to obtain
		\begin{align*}
			I_{4} (\tau) \leq C|1-a|, \quad C_{\omega} (\tau) \leq 2|1-a|^2,
		\end{align*}
		where $C>0$ is some absolute constant.

Note that
		\begin{align*}
			\bar{c}_{\omega}=(1-a)+\nu C_{\omega}(\tau), \quad \hat{c}_{\omega}(\tau) = (1-a)\hat{u}_{x}(0,\tau) + \nu C_{\omega}(\tau) \hat{\omega}_{xxx}(0,\tau).
		\end{align*}
		By the Sobolev embedding, we have
		\begin{align*}
			|\hat{u}_{x} (0)| \lesssim \|\hat{\omega}\|_{\mathcal{H}} \lesssim I_{4} (\tau) \lesssim |1-a|
		\end{align*}
		and
		\begin{align*}
			|\hat{\omega}_{xxx}(0,\tau)| \lesssim I_{4} (\tau) \lesssim |1-a|.
		\end{align*}
		It concludes that there exists a small number $0<\delta\le \delta_3$ such that
		\begin{align*}
			c_{\omega}(\tau) & = \bar{c}_{\omega} + \hat{c}_{\omega} (\tau)\\
			& \leq (1-a) + C[(1-a)^{2} + \nu (1+|1-a|)(1-a)^{2}]<0,\\
		\end{align*}
		for $a\in (1,1+\delta)$.
		It follows that  blow-up occurs in finite time in the original physical variables for the viscous model \ref{nuPJE} with $a\in (1,1+\delta)$, as stated in Theorem \ref{the viscous}.

The proof of Theorem \ref{the viscous} is finished.
	\end{proof}

	\section{Appendix}\label{section appendix}
	 In the Appendix, we provide the proof of Lemma \ref{le operator}. To extract the maximal damping effect, we expand the perturbed solution in a Fourier series and perform exact calculations.
	 \begin{lemma}\label{le:linear operator}
	 	It holds that
	 	\begin{align*}
	 		\left\langle\mathcal{L}_{1} \hat{\omega},\hat{\omega} \right\rangle_{\mathcal{H}} \leq -\frac{1}{2}\|\hat{\omega}\|_{\mathcal{H}}^{2}.
	 	\end{align*}
	 \end{lemma}
	\begin{proof}[Proof of Lemma \ref{le:linear operator}]
		Note that
		\begin{align}
			\mathcal{L}_{1}e_{k}^{(o)} = A_{k}e_{k+1}^{(o)} + B_{k}e_{k-1}^{(o)}, \quad  k \geq 2,
		\end{align}
		where
		\begin{align*}
			e_{k}^{(o)} = \sin(kx), \quad k \geq 1
		\end{align*}
		and
		\begin{align}
			A_{k} = -\frac{(k+1)(k-1)^{2}}{2k^{2}}, \quad B_{k} = \frac{(k+1)^{2}(k-1)}{2k^{2}}, \quad k \geq 2.
		\end{align}
		For $k \geq 2,$ direct calculations give
		\begin{align*}
			\mathcal{L}_{1}\tilde{e}_{k}^{(o)} = & \frac{A_{k+1}}{k+1}e_{k+2}^{(o)} + \frac{B_{k+1}}{k+1}e_{k}^{(o)} - \frac{A_{k}}{k}e_{k+1}^{(o)} - \frac{B_{k}}{k}e_{k-1}^{(o)}\\
			= & -\frac{(k+2)k^{2}}{2(k+1)^{3}}e_{k+2}^{(o)} + \frac{(k+2)^{2}k}{2(k+1)^{3}}e_{k}^{(o)} + \frac{(k+1)(k-1)^{2}}{2k^{3}}e_{k+1}^{(o)} - \frac{(k+1)^{2}(k-1)}{2k^{3}}e_{k-1}^{(o)}\\
			= & -\frac{(k+2)^{2}k^{2}}{2(k+1)^{3}}\left(\frac{e_{k+2}^{(o)}}{k+2} - \frac{e_{k+1}^{(o)}}{k+1}\right) - \frac{(k+2)^{2}k^{2}}{2(k+1)^{3}}\frac{e_{k+1}^{(o)}}{k+1} + \frac{(k+2)^{2}k^{2}}{2(k+1)^3}\frac{e_{k}^{(o)}}{k}\\
			& +  \frac{(k+1)^{2}(k-1)^{2}}{2k^{3}}\left(\frac{e_{k}^{(o)}}{k} - \frac{e_{k-1}^{(o)}}{k-1}\right) - \frac{(k+1)^{2}(k-1)^{2}}{2k^{3}}\frac{e_{k}^{(o)}}{k} + \frac{(k+1)^{2}(k-1)^{2}}{2k^{3}}\frac{e_{k-1}^{(o)}}{k-1}\\
			:= & -d_{k+1}\tilde{e}_{k+1}^{(o)} + (-d_{k+1} + d_{k})\tilde{e}_{k}^{(o)} + d_{k}\tilde{e}_{k-1}^{(o)}.
		\end{align*}
		In addition, when $k=1$, we have
		\begin{align*}
			\mathcal{L}_{1} e_{1}^{(o)} = 0
		\end{align*}
		and
		\begin{align*}
			\mathcal{L}_{1}\tilde{e}_{1}^{(o)}=-d_{2}\tilde{e}_{2}^{(o)}+(d_{1}-d_{2})\tilde{e}_{1}^{(o)}.
		\end{align*}
		Thus we can write
		\begin{align*}
			\mathcal{L}_{1}\tilde{e}_{k}^{(o)} = -d_{k+1}\tilde{e}_{k+1}^{(o)} + (-d_{k+1} + d_{k})\tilde{e}_{k}^{(o)} + d_{k}\tilde{e}_{k-1}^{(o)},
		\end{align*}
		where $d_{k} = \frac{(k+1)^{2}(k-1)^{2}}{2k^{3}}, k = 1, 2, \cdots.$
		If we expand \(\hat{\omega} (x,\tau)\) as
		\[\hat{\omega}(x, \tau) = \sum_{k \geq 1} \hat{\omega}^{(o)}_k(\tau)\, \tilde{e}^{(o)}_k,\]
		then the equation $\partial_{\tau} \hat{\omega} = \mathcal{L}_{1} \hat{\omega}$ reduces to the following infinite-dimensional system of ordinary differential equations
		\begin{align*}
			\partial_{\tau} \hat{\omega}_{k}^{(o)} (\tau) = - d_{k} \hat{\omega}_{k-1}^{(o)} (\tau) - (d_{k+1} - d_{k}) \hat{\omega}_{k}^{(o)} (\tau) + d_{k+1} \hat{\omega}_{k+1}^{(o)} (\tau), \quad k \geq 1,
		\end{align*}
		where $d_{1} \hat{\omega}_{0}^{(o)}$ is understood to be zero. Thus, we formally deduce that
		\begin{align}\label{sumomegak}
			\begin{aligned}
				\frac{1}{2} \partial_{\tau} \sum_{k\geq 1} \big(\hat{\omega}_{k}^{(o)}\big)^{2} & =
				\left\langle \mathcal{L}_{1}\hat{\omega},\hat{\omega} \right\rangle_{\mathcal{H}}\\
				& = \sum_{k\geq 1} -d_{k} \hat{\omega}_{k-1}^{(o)}\hat{\omega}_{k}^{(o)} + (-d_{k+1} + d_{k}) \big(\hat{\omega}_{k}^{(o)}\big)^{2} + d_{k+1} \hat{\omega}_{k}^{(o)}\hat{\omega}_{k+1}^{(o)}\\
				& = \sum_{k\geq 1} (-d_{k+1} + d_{k}) \big(\hat{\omega}_{k}^{(o)}\big)^{2}\\
				&\leq  - \frac{1}{2} \big(\hat{\omega}_{k}^{(o)}\big)^{2}.
			\end{aligned}
		\end{align}
		Here, we have used the fact that for all $k\geq 1$,
		\begin{align*}
			d_{k+1} - d_{k} &= \frac{(k+2)^{2}k^{2}}{2(k+1)^{3}} -  \frac{(k+1)^{2}(k-1)^{2}}{2k^{3}}\\
			&=\frac{k^{6} + 3k^{5} +5k^{4} + 5k^{3} - k^{2} -3k -1}{2k^{3}(k+1)^{3}}\\
			&=\frac{1}{2} + \frac{2k^{4} + 4k^{3} -k^{2} -3k -1}{2k^{3}(k+1)^{3}}\\
			& > \frac{1}{2}.
		\end{align*}
		While the formal computation in \ref{sumomegak} illustrates the dissipative structure of \(\mathcal{L}_1\), it is required to justify the validity of the summation, which may not converge in a straightforward manner. One rigorous approach is to invoke standard linear semigroup theory.
		To this end, we consider the real Hilbert space $Y$ formally spanned by the basis functions $\{\tilde{e}_{k}^{(o)}, k\geq 1\}$ in which this basis is orthonormal, namely
		\begin{align*}
			Y = \bigg\{\hat{\omega} (x,\tau) = \sum_{k\geq 1} \hat{\omega}_{k}^{(o)} (\tau) \tilde{e}_{k}^{(o)} \bigg| \{\hat{\omega}_{k}^{(o)}\}_{k\geq 1}\in l^{2}\bigg\}.
		\end{align*}
		The operator \(\mathcal{L}_1\) defines a densely defined, closed, unbounded linear operator on \(Y\). The estimate \eqref{sumomegak}, together with standard energy methods, implies via a direct application of the Hille--Yosida theorem that \(\mathcal{L}_1\) generates a strongly continuous semigroup satisfying
		\begin{align*}
			\|e^{\tau \mathcal{L}_{1}} \hat{\omega} (0)\|_{\mathcal{H}} \leq e^{-\frac{1}{2} \tau} \|\hat{\omega} (0)\|_{\mathcal{H}}.
		\end{align*}
		An alternative approach is to apply Galerkin's method to establish global well-posedness. In addition, the decay rate of the Fourier coefficients of the solution can be rigorously derived (see \cite{GJ2025} for more details).
	\end{proof}


\begin{thebibliography}{99}
		
		\bibitem{B1948}
		J.~M. Burgers.
		\newblock A mathematical model illustrating the theory of turbulence.
		\newblock In R.~von Mises and T.~von K\'arm\'an, {Advances in
			{A}pplied {M}echanics}, pages 171--199. Academic Press, New York, 1948.
		
		
		\bibitem{CF1984}
		F.~Calogero.
		\newblock A solvable nonlinear wave equation.
		\newblock {Stud. Appl. Math.}, 70(3):189--199, 1984.
		
		\bibitem{C2021}
		J.~Chen.
		\newblock On the slightly perturbed {D}e {G}regorio model on {$S^1$}.
		\newblock {Arch. Ration. Mech. Anal.}, 241(3):1843--1869, 2021.
		
		\bibitem{CH2021}
		J.~Chen and T.~Y. Hou.
		\newblock Finite time blowup of 2{D} {B}oussinesq and 3{D} {E}uler equations
		with {$C^{1,\alpha}$} velocity and boundary.
		\newblock {Comm. Math. Phys.}, 383(3):1559--1667, 2021.
		
		\bibitem{CHH2021}
		J.~Chen, T.~Y. Hou, and D.~Huang.
		\newblock On the finite time blowup of the {D}e {G}regorio model for the 3{D}
		{E}uler equations.
		\newblock {Comm. Pure Appl. Math.}, 74(6):1282--1350, 2021.
		
		\bibitem{CS1989}
		S.~Childress, G.~R. Ierley, E.~A. Spiegel, and W.~R. Young.
		\newblock Blow-up of unsteady two-dimensional {E}uler and {N}avier-{S}tokes
		solutions having stagnation-point form.
		\newblock {J. Fluid Mech.}, 203:1--22, 1989.
		
		\bibitem{CW2010}
		C.-H. Cho and M.~Wunsch.
		\newblock Global and singular solutions to the generalized {P}roudman-{J}ohnson
		equation.
		\newblock {J. Differential Equations}, 249(2):392--413, 2010.
		
		\bibitem{CW2012}
		C.-H. Cho and M.~Wunsch.
		\newblock Global weak solutions to the generalized {P}roudman-{J}ohnson
		equation.
		\newblock {Commun. Pure Appl. Anal.}, 11(4):1387--1396, 2012.
		
		\bibitem{CW2009}
		A.~Constantin and M.~Wunsch.
		\newblock On the inviscid {P}roudman-{J}ohnson equation.
		\newblock {Proc. Japan Acad. Ser. A Math. Sci.}, 85(7):81--83, 2009.
		
		\bibitem{E2021}
		T.~Elgindi.
		\newblock Finite-time singularity formation for {$C^{1,\alpha}$} solutions to
		the incompressible {E}uler equations on {$\Bbb R^3$}.
		\newblock {Ann. of Math. (2)}, 194(3):647--727, 2021.
		
		\bibitem{EG2021}
		T.~M. Elgindi, T.-E. Ghoul, and N.~Masmoudi.
		\newblock Stable self-similar blow-up for a family of nonlocal transport
		equations.
		\newblock {Anal. PDE}, 14(3):891--908, 2021.
		
		\bibitem{GJ2025}
		J.~Guo and Q.~Jiu.
		\newblock Stability and Instability on the De Gregorio Modification of the
		Constantin-Lax-Majda model.
		\newblock {arXiv preprint arXiv:2506.02800}, 2025.
		
		\bibitem{HW2024}
		T.~Y. Hou and Y.~Wang.
		\newblock Blowup analysis for a quasi-exact 1{D} model of 3{D} {E}uler and
		{N}avier-{S}tokes.
		\newblock {Nonlinearity}, 37(3):Paper No. 035001, 28, 2024.
		
		\bibitem{B1991}
		J.~K. Hunter and R.~Saxton.
		\newblock Dynamics of director fields.
		\newblock {SIAM J. Appl. Math.}, 51(6):1498--1521, 1991.
		
		\bibitem{KM2006}
		C.~E. Kenig and F.~Merle.
		\newblock Global well-posedness, scattering and blow-up for the
		energy-critical, focusing, non-linear {S}chr\"odinger equation in the radial
		case.
		\newblock {Invent. Math.}, 166(3):645--675, 2006.
		
		\bibitem{KF2020}
		F.~Kogelbauer.
		\newblock On the global well-posedness of the inviscid generalized
		{P}roudman-{J}ohnson equation using flow map arguments.
		\newblock {J. Differential Equations}, 268(3):1050--1080, 2020.
		
		\bibitem{LP1988}
		M.~J. Landman, G.~C. Papanicolaou, C.~Sulem, and P.-L. Sulem.
		\newblock Rate of blowup for solutions of the nonlinear {S}chr\"odinger
		equation at critical dimension.
		\newblock {Phys. Rev. A (3)}, 38(8):3837--3843, 1988.
		
		\bibitem{LL2020}
		Z.~Lei, J.~Liu, and X.~Ren.
		\newblock On the {C}onstantin-{L}ax-{M}ajda model with convection.
		\newblock {Comm. Math. Phys.}, 375(1):765--783, 2020.
		
		\bibitem{L2007}
		J.~Lenells.
		\newblock The {H}unter-{S}axton equation describes the geodesic flow on a
		sphere.
		\newblock {J. Geom. Phys.}, 57(10):2049--2064, 2007.
		
		\bibitem{L2008}
		J.~Lenells.
		\newblock The {H}unter-{S}axton equation: a geometric approach.
		\newblock {SIAM J. Math. Anal.}, 40(1):266--277, 2008.
		
		\bibitem{LZ2022}
		C.~Li and K.~Zhang.
		\newblock A note on the {G}agliardo-{N}irenberg inequality in a bounded domain.
		\newblock {Commun. Pure Appl. Anal.}, 21(12):4013--4017, 2022.
		
		\bibitem{MM2014}
		Y.~Martel, F.~Merle, and P.~Rapha\"el.
		\newblock Blow up for the critical generalized {K}orteweg--de {V}ries equation.
		{I}: {D}ynamics near the soliton.
		\newblock {Acta Math.}, 212(1):59--140, 2014.
		
		\bibitem{MP1986}
		D.~W. McLaughlin, G.~C. Papanicolaou, C.~Sulem, and P.~L. Sulem.
		\newblock Focusing singularity of the cubic schr\"odinger equation.
		\newblock {Phys. Rev. A}, 34:1200--1210, Aug 1986.
		
		\bibitem{MR2005}
		F.~Merle and P.~Raphael.
		\newblock The blow-up dynamic and upper bound on the blow-up rate for critical
		nonlinear {S}chr\"odinger equation.
		\newblock {Ann. of Math. (2)}, 161(1):157--222, 2005.
		
		\bibitem{MZ1997}
		F.~Merle and H.~Zaag.
		\newblock Stability of the blow-up profile for equations of the type
		{$u_t=\Delta u+|u|^{p-1}u$}.
		\newblock {Duke Math. J.}, 86(1):143--195, 1997.
		
		\bibitem{OH2009}
		H.~Okamoto.
		\newblock Well-posedness of the generalized {P}roudman-{J}ohnson equation
		without viscosity.
		\newblock {J. Math. Fluid Mech.}, 11(1):46--59, 2009.
		
		\bibitem{OO2005}
		H.~Okamoto and K.~Ohkitani.
		\newblock On the role of the convection term in the equations of motion of
		incompressible fluid.
		\newblock {Journal of the Physical Society of Japan}, 74(10):2737--2742,
		2005.
		
		\bibitem{OS2008}
		H.~Okamoto, T.~Sakajo, and M.~Wunsch.
		\newblock On a generalization of the {C}onstantin-{L}ax-{M}ajda equation.
		\newblock {Nonlinearity}, 21(10):2447--2461, 2008.
		
		\bibitem{OHZ2000}
		H.~Okamoto and J.~Zhu.
		\newblock Some similarity solutions of the {N}avier-{S}tokes equations and
		related topics.
		\newblock In {Proceedings of 1999 {I}nternational {C}onference on
			{N}onlinear {A}nalysis ({T}aipei)}, volume~4, pages 65--103, 2000.
		
		\bibitem{P2001}
		M.~V. Pavlov.
		\newblock The {C}alogero equation and {L}iouville-type equations.
		\newblock {Teoret. Mat. Fiz.}, 128(1):109--115, 2001.
		
		\bibitem{SS2013}
		A.~Sarria and R.~Saxton.
		\newblock Blow-up of solutions to the generalized inviscid {P}roudman-{J}ohnson
		equation.
		\newblock {J. Math. Fluid Mech.}, 15(3):493--523, 2013.
		
		\bibitem{SS2015}
		A.~Sarria and R.~Saxton.
		\newblock The role of initial curvature in solutions to the generalized
		inviscid {P}roudman-{J}ohnson equation.
		\newblock {Quart. Appl. Math.}, 73(1):55--91, 2015.
		
		\bibitem{W2011}
		M.~Wunsch.
		\newblock The generalized {P}roudman-{J}ohnson equation revisited.
		\newblock {J. Math. Fluid Mech.}, 13(1):147--154, 2011.
		
		\bibitem{Y2004}
		Z.~Yin.
		\newblock On the structure of solutions to the periodic {H}unter-{S}axton
		equation.
		\newblock {SIAM J. Math. Anal.}, 36(1):272--283, 2004.
		
	\end{thebibliography}
\end{document}